\documentclass[a4paper,english,10pt]{amsart}
\usepackage[applemac]{inputenc}
\usepackage{inputenc}
\usepackage{babel}
\usepackage{amsfonts}
\usepackage{amsmath}
\usepackage{amssymb}
\usepackage{amsthm}
\usepackage{calligra,mathrsfs}
\usepackage{graphicx}
\usepackage[all]{xy}
\usepackage{textcomp}
\usepackage{tikz-cd}
\usepackage{enumitem} 
\usepackage{hyperref}
\hypersetup{
  colorlinks = true,
     urlcolor =   blue,
    linkcolor =  blue,
     citecolor = purple}

\newtheorem{thm}{Theorem}[section]  
\newtheorem{cor}[thm]{Corollary}
\newtheorem{lem}[thm]{Lemma}
\newtheorem{defi}[thm]{Definition}
\newtheorem{prop}[thm]{Proposition}
\newtheorem{es}[thm]{Example}
\newtheorem{rem}[thm]{Remark}

\DeclareMathOperator{\Homs}{\mathscr{H}\text{\kern -3pt {\calligra\large om}}\,}

\title{Radiality of definable sets}
\author{John Welliaveetil}

\begin{document}

\maketitle

\begin{abstract} 
  In this article we use techniques developed by Hrushovski-Loeser to study certain metric properties of the 
  Berkovich analytification of a finite morphism of smooth connected projective curves. In recent work, M. Temkin proved a radiality statement 
  for the topological ramification locus associated to such finite morphisms. We generalize this result in two 
  directions. We prove a radiality statement for a more general class of sets which we 
  call definable sets. In another direction, we show that the result of Temkin can be obtained in families. 
\end{abstract} 

 {\hypersetup {linkcolor = black} 
\tableofcontents
}
\section{Introduction} \label{introduction}
 
      It is one of the many redeeming features of non-Archimedean geometry that the
      analytification of an algebraic curve is endowed with a structure 
      that can be described in fairly explicit terms. 
      For example, if $k$ is an algebraically closed field complete
      with respect to a non-trivial non-Archimedean real valuation then every smooth
       curve $C$ over $k$ posseses a \emph{skeleton} $\Sigma$\footnote{Refer \cite[\S 3.5.1]{TEM2} for a definition of skeleton.}.
      In particular, $C^{\mathrm{an}} \smallsetminus \Sigma$ decomposes into the disjoint union of isomorphic copies
      of the Berkovich open unit disk and $C^{\mathrm{an}}$ admits a deformation retraction onto $\Sigma$ (cf. \cite[Lemma 3.4]{BPR}).

       Skeleta can also be used to shed light on the structure of a finite morphism $f : C' \to C$ of 
       smooth $k$-curves.  
        For instance, it is a well known fact that there exists skeleta $\Sigma \subset C^{\mathrm{an}}$ and 
       $\Sigma' \subset C'^{\mathrm{an}}$ such that 
       $\Sigma' = (f^{\mathrm{an}})^{-1}(\Sigma)$ and the deformation retractions 
       of $C'^{\mathrm{an}}$ onto $\Sigma'$ and $C^{\mathrm{an}}$ onto $\Sigma$ 
       can be taken to be compatible for the morphism $f^{\mathrm{an}}$ 
        (cf. \cite[\S 3]{WE2}). 
       Since $f^{\mathrm{an}}$ is clopen \footnote{This is a particular instance
       of a more general statement (cf. \cite[Lemma 2.17]{WE1},\cite[Proposition 3.4.7]{berk90}).}, we see that if 
       $O'$ is a connected component of $C'^{\mathrm{an}} \smallsetminus \Sigma'$ 
       then there exists a connected component $O$ of $C^{\mathrm{an}} \smallsetminus \Sigma$ such that 
       the morphism $f^{\mathrm{an}}$ restricts to a surjective map 
       $O' \to O$. If we were to identify $O'$ and $O$ with the Berkovich open unit disk $O(0,1)$ then 
       outside $\Sigma'$, the morphism $f^{\mathrm{an}}$ reduces to finite endomorphisms of
       $O(0,1)$.        
       
            Let $x \in O'$ be a point of type I, II or III.
      The point $x$ must correspond to a closed sub-ball around a rigid point 
      $a$ of radius $r$.  
         We use the notation from \cite[\S 1.2]{BR10} and 
         write $x$ as $\zeta_{a,r}$ where $a \in O'(k)$, $r \in [0,1)$.
       Having identified $O'$ with the Berkovich open unit disk, we define the map 
       \begin{align*} 
       l_x : (0,-\mathrm{log}(r)] \to O' \\
       t \mapsto \zeta_{a,\mathrm{exp}(-t)}
       \end{align*}
        where 
       $\zeta_{a,\mathrm{exp}(-t)}$ is the type II point
       corresponding to the closed disk around $x$ of radius $\mathrm{exp}(-t)$. 
       By convention, we regard any interval 
       $[t',t] \subset (0,\infty]$ as starting from $t$
       with terminal point $t'$.
       We set $|l_x| := -\mathrm{log}(r)$.
       We can extend these constructions to treat points of type IV as well. 

         Note that for every Zariski closed point
         \footnote{These are also referred to as rigid points.}
          $x \in O'(k)$, the map 
        $l_x$ is an isometric embedding with respect to the metric on $C'^{\mathrm{an}}$. 
        It follows that $f^{\mathrm{an}}$ induces a map 
        $$f^{\mathrm{an}}_x := l_{f^{\mathrm{an}}(x)}^{-1} \circ f^{\mathrm{an}}  \circ l_x : (0,\infty] \to (0,\infty].$$    
                We deviate slightly from the notation introduced in \cite{TEM3} and refer to 
        $f^{\mathrm{an}}_x$ as the \emph{profile function of $f$ at $x$}. 
        Note that the profile functions are dependent on the skeleta chosen. 
        The morphism 
        $f^{\mathrm{an}}_x$ is continuous, piecewise 
        linear and its slopes are positive integers.
         It 
        follows that there exists $m_x \in \mathbb{N}$ and
        tuples
        $$T_x := ((t_0,d_0,\alpha_0), \ldots, (t_{m_x},d_{m_x},\alpha_{m_x})) \subset ((0,\infty] \times \mathbb{N} \times \mathrm{val}(k^*))^{m_x}$$ 
         such that 
         $t_0 = \infty$ and for every $0 \leq i \leq m_x$, 
         $f^{\mathrm{an}}_x$ restricted to $(t_{i+1},t_i]$ coincides with
         the map
         $s \mapsto d_is + \alpha_i$.
         By $\mathrm{val}$, we mean $-\mathrm{log}|\cdot|$. 
         The points $t_i$ are the \emph{break points} of 
         $f^{\mathrm{an}}_x$ i.e. 
         for every $i$, $d_i \neq d_{i+1}$. 
               Observe that 
          $f^{\mathrm{an}}_x$ is completely determined by the 
          integer $m_x$ and the tuple $T_x := ((t_0,d_0,\alpha_0), \ldots, (t_{m_x},d_{m_x},\alpha_{m_x}))$. 
       Furthermore, $f^{\mathrm{an}}_x$ is independent of how 
       we identified $O'$ with the Berkovich unit disk 
       since $l_x$ is an isometric embedding. 
       
          Given the curve $C'$ and the skeleton $\Sigma'$, 
          we introduce a function 
          $r_{\Sigma'} : C'^{\mathrm{an}} \to [0,\infty]$ as in 
          \cite[\S 3.2.2]{TEM3} which we require when we define the radiality of a set. 
        If $x \in C'^{\mathrm{an}} \smallsetminus \Sigma'$ then we set 
        $r_{\Sigma'}(x) := |l_x|$. 
        If $x \in \Sigma'(k)$ then $r_{\Sigma'}(x) := \infty$
        and if $x \in \Sigma'$ is not a $k$-point then $r_{\Sigma'}(x) := 0$.
              In \cite{TEM3}, M. Temkin proved the following striking result. 
       He showed that there exists a skeleton $(\Sigma',\Sigma)$ for the  
       morphism $f^{\mathrm{an}}$ \footnote{This means that $\Sigma' \subset C'^{\mathrm{an}}$ and 
       $\Sigma \subset C^{\mathrm{an}}$ are skeleta and the retraction maps 
       $C'^{\mathrm{an}} \to \Sigma'$ and $C^{\mathrm{an}} \to \Sigma$ are compatible 
       with $f^{\mathrm{an}}$.} such that
       for every $x \in C'(k)$, the tuple $T_x$ depends only on its 
       image in $\Sigma'$ for the retraction map $\tau'_{\Sigma'} : C'^{\mathrm{an}} \to \Sigma'$. 
        Equivalently, if $x$ and $y$ belong to $C'(k) \smallsetminus \Sigma'$ are such that  
        $\tau'_{\Sigma'}(x) = \tau'_{\Sigma'}(y)$ 
        then $m_x = m_y$ and $T_x = T_y$.
        As a consequence, one deduces as in Theorem 3.2.10 in loc.cit. that for every $d \in \mathbb{N}$, the set $N_{f^{\mathrm{an}},d}$ of
       points of multiplicity atleast 
        $d$ (cf. \cite[\S 2.1.3]{TEM3} is \emph{radial} around $\Sigma'$. 
        
        \begin{defi} \label{radial set} 
        \emph{Let $C$ be a smooth projective $k$-curve and $\Sigma_0$ 
        be a skeleton of $C^{\mathrm{an}}$.
        We say a subset $X \subset C^{\mathrm{an}}$ is \emph{radial} around the skeleton $\Sigma_0 \subset C^{\mathrm{an}}$
        if there exists a function $p_X : \Sigma_0 \to \mathbb{R}_\infty$ such that   
        $x \in C^{\mathrm{an}}$ belongs to $X$ if and only if $r_{\Sigma_0}(x) \leq p_X(\gamma)$ 
        where $\gamma \in \Sigma_0$ is the unique point that $x$ retracts onto.
        We will say that $X$ is} piecewise affine radial \emph{if in addition, the function 
        $p_X$ is piecewise affine.} 
        \end{defi} 
        
           Observe that since the paths $l_x$ are isometries, a radial set around a skeleton $\Sigma_0$ is a metric neighbourhood 
           of $\Sigma_0$ where the distance to the skeleton is given by a function on $\Sigma_0$ (cf. \cite[3.2.4]{TEM3}).  
           
       \subsection{Radiality of definable sets} 
           
        Motivated by the structure of the multiplicity loci $N_{f^{\mathrm{an}},d}$, one asks 
        if the phenomenon of radiality is displayed by members of a more general class of sets. 
        To this end we introduce the notion of a \emph{definable} subset of $C^{\mathrm{an}}$ and show that 
          such subsets are in fact radial around suitable skeleta. 
        Furthermore, one can check easily that the multiplicity loci are examples of 
        definable subsets of $C'^{\mathrm{an}}$. 
       
        Strictly speaking, one cannot discuss the definability in the model theoretic sense
        of a subset of $C'^{\mathrm{an}}$. However, the theory of non-Archimedean geometry as 
        introduced by Hrushovski-Loeser allows us to overcome this hurdle. 
            In order to  better understand the homotopy types of 
          the Berkovich analytification of quasi-projective varieties, Hrushovski and Loeser 
          introduced a model theoretic analogue of the Berkovich space. 
          To apply model theory to the study of non-Archimedean geometry, 
          we work in the theory ACVF of algebraically closed valued fields. 
           The Hrushovski-Loeser space associated to a 
           $k$-curve $C$ is the space of stably dominated types that 
           concentrate on $C$ and is denoted $\widehat{C}$.    
          A precise definition and detailed description of 
          $\widehat{C}$ can be found in \cite[\S 2]{HL}. 
          The space $\widehat{C}$ is a Hausdorff topological space and can be seen 
          as an enrichment of $C$.  
         
          Let us emphasize two important features that $\widehat{C}$ is 
          endowed with. 
          Firstly, the space $\widehat{C}$ is a $k$-definable set (cf. \cite[\S 7.1]{HL}). 
           We can relate $\widehat{C}$ to the Berkovich space 
           $C^{\mathrm{an}}$ as follows. 
           Let $k^{max}$ be a maximally complete, algebraically closed 
           field extension of $k$
           whose value group is $\mathbb{R}^*$ (when the value group is written multiplicatively). 
           Such a field is unique up to isomorphism over $k$. 
          There exists a canonical homeomorphism 
          $\pi : \widehat{C}(k^{max}) \to C^{\mathrm{an}}_{k^{max}}$.
          This induces a map 
         $\pi_{k,C} : \widehat{C}(k^{max}) \to C^{\mathrm{an}}$ by composing 
         $\pi$ with the projection $C^{\mathrm{an}}_{k^{max}} \to C^{\mathrm{an}}$.
          The map $\pi_{k,C}$ is closed and continuous. 
          
          \begin{defi} \label{definable subset of an analytic curve}
           \emph{A set $X \subset C^{\mathrm{an}}$ is } $k$-definable 
            \emph{if there exists a $k$-definable subset 
            $Y \subset \widehat{C}$ such that 
            $X = \pi_{k,C}(Y)$. We say that $X$ is closed if the set $Y$ is closed and $X$ is path-connected 
            if $Y$ is path connected.}           
          \end{defi} 
           
  \begin{thm} \label{definable subsets are radial}
   Let $C$ be a smooth projective curve over $k$. 
   We then have that every closed path connected $k$-definable set 
   is piecewise affine radial.  
    \end{thm}
    
      Theorem \ref{definable subsets are radial} is the consequence of an analogous result for 
     Hrushovski-Loeser curves. 
     In order to discuss the radiality statement for definable subsets of $\widehat{C}$, 
     we introduce the notion of a definable homotopy. 
     
       A \emph{definable homotopy} on the curve $C$ is a continuous function
  $$h : [0,\infty] \times C \to \widehat{C}.$$ 
  In \cite[\S 7.5]{HL}, Hrushovski and Loeser construct such homotopies such that the 
  image $h(0,C) \subset \widehat{C}$ is $\Gamma$-internal.
  One may think of a $\Gamma$-internal subset of $\widehat{C}$ as the model theoretic 
  analogue of a subspace of $C^{\mathrm{an}}$ homeomorphic to a finite simplicial complex. 
  Henceforth, we assume that all definable homotopies under consideration have $\Gamma$-internal
  subsets as image. 
  We discuss the precise nature of these constructions briefly in \S \ref{homotopies of C}.
     
     Let $h$ be a definable homotopy of ${C}$ onto a $\Gamma$-internal set $\Sigma$. We
  associate to every definably path-connected definable subset of $\widehat{C}$
  a function $\alpha : C \to [0,\infty] \times \{\gamma_1,\gamma_2\}$
  where $\gamma_1 < \gamma_2 \in \Gamma_\infty[k]$.
  Although strictly speaking one cannot talk about metric properties of the curve 
  $\widehat{C}$, the function $\alpha$ in some sense measures the distance from the definable set 
  to $\Sigma$. Theorem \ref{radiality of definable sets} shows that
   given a definably path connected subset of $\widehat{C}$, 
   we can choose $h$ and $\Sigma$ so 
  the associated function $\alpha$ is constant along the fibres of the retraction associated 
  to $h$. 

\subsection{Backward branching index}  
  An important ingredient in the proof of Theorem \ref{definable subsets are radial} is the 
  finiteness property satisfied by the profile functions introduced above. By this we mean that
  if we are given a morphism $f : C' \to C$ then there exists skeleta 
  $\Sigma'$ and $\Sigma$ compatible for $f^{\mathrm{an}}$ 
  and such that $f_{-}^{\mathrm{an}}$ is constant along the fibres of the retractions 
  i.e. if $x,y \in C'(k)$ retract to the same point on $\Sigma'$ then 
  $f_x^{\mathrm{an}} = f_{y}^{\mathrm{an}}$. This is one of the main results 
  of Temkin in \cite{TEM3}.
  Our approach however is different in that we introduce an analogue of the profile functions 
  $f_x^{\mathrm{an}}$ within the framework of the Hrushovski-Loeser curves. 
  This new invariant is called the \emph{Backward branching index function} associated 
  to a definable homotopy.

     We reproduce the definition from \S \ref{backward branching index}.
    We suppose that the definable homotopy 
    $h : [0,\infty] \times C \to \widehat{C}$ lifts uniquely to a
    a definable homotopy $h' : [0,\infty] \times C' \to \widehat{C'}$ via $\widehat{f}$. 
       Let $\Sigma'$ denote the
        $\Gamma$-internal set which is the
         image of $h'$. 
          
          \begin{defi} The backward branching index of a point $x \in C'$ \emph{is a function}
          $$BB_{h',f}(x) : [0,\infty] \to \mathbb{N} \times [0,\infty]$$
         \emph{defined as follows. 
         If $t \in [0,\infty]$, let $B_{h',f}(x)(t)$ denote the set of 
          elements $z \in C'$ such that 
        $f(z) = f(x)$, 
        $h'(t,z) = h'(t,x)$ and $h'(t,z) \neq h'(0,x)$.
        Then 
        \begin{align*}
        BB_{h',f}(x)(t) := (\mathrm{card}(B_{h',f}(x)(t)), \Sigma_{z \in B_{h',f}(x)(t)} \mathrm{val}(z)).
        \end{align*}
        }
    
        \end{defi}
        
        By Remark \ref{definability of backward branching index}, 
        we
        thus have a function 
        $$BB_{h',g} : C \to \mathrm{Fn}([0,\infty],\mathbb{N} \times [0,\infty])$$
         where  $\mathrm{Fn}([0,\infty],\mathbb{N} \times [0,\infty])$ is the set of definable
        functions $[0,\infty] \to \mathbb{N} \times [0,\infty]$.
        
        In Theorem \ref{backward branching is constant along retractions}, we 
        show that the backward branching index functions can be made to satisfy a strong 
        finiteness property. More precisely, we show that
        we can choose $h'$ and $h$ so that if 
        $x,y \in C'(k)$ retract to the same point in $\Sigma'$ then
        $BB_{h',f}(x) = BB_{h',f}(y)$. The proof of this fact 
        exploits the explicit nature and flexibility with which Hrushovski and Loeser 
        construct definable homotopies on $C'$ and $C$.
        We relate the Backward branching indices to the 
        profile functions in Lemma \ref{T_x and S_x are intercalculable} and
        proceed from there to deduce in Theorem \ref{tameness of profile functions}
         the finiteness property discussed above
        for the profile functions.
        
 \subsection{Tameness of profile functions in families}

      By introducing the Backward branching index invariants and proving that 
      they determine the profile functions, we have set ourselves up to  
      prove a generalization of Theorem \ref{tameness of profile functions}. 
       In Theorem \ref{tameness of profile functions in families}, we
       study Theorem \ref{tameness of profile functions} in the context of 
       families of morphisms of smooth projective curves. 
       The Theorem \ref{tameness of profile functions} shows that 
        considerable information about the metric properties of a finite morphism can 
        be encoded in terms of combinatorial data i.e.
         the induced morphism of suitable skeleta and additional data on the skeleton of the 
        source that controls the profile functions we introduced earlier. 
        The goal of Theorem \ref{tameness of profile functions in families} is 
         to show that given a family of morphisms of smooth projective curves 
         the associated family of combinatorial data itself satisfies a nice finiteness property.

       The framework within which we work is the following. 
       Let $S$ be a quasi-projective $k$-variety.
       Let $\alpha_1 : X \to S$ and $\alpha_2 : Y \to S$ be smooth morphisms of quasi-projective $k$-varieties.  
 Let $f : X \to Y$ be a morphism such that for every $s \in S$, 
  $f_s : X_s \to Y_s$ is a finite separable morphism of smooth projective irreducible $k$-curves.
  
  The results that we discussed previously, notably Theorem 
  \ref{tameness of profile functions}, 
  imply that 
  for every $s \in S(k)$, 
  there exist deformation retractions 
  $$h_s'^{\mathrm{an}} : [0,\infty] \times X_s^{\mathrm{an}} \to X_s^{\mathrm{an}}$$ and 
  $$h_s^{\mathrm{an}} : [0,\infty] \times Y_s^{\mathrm{an}} \to Y_s^{\mathrm{an}}$$
  such that the following hold. 
   \begin{enumerate} 
  \item The images $\Sigma_s' := {h'}_s^{\mathrm{an}}(0,X_s^{\mathrm{an}})$ and $\Sigma_s = h_s^{\mathrm{an}}(0,Y_s^{\mathrm{an}})$ are
   skeleta of the curves $X_s^{\mathrm{an}}$ and $Y_s^{\mathrm{an}}$. 
  \item 
  Given $x \in X_s(k)$, let $\mathrm{prfl}(f_s^{\mathrm{an}},x)$ denote the profile function at $x$
  associated to $f_s$.
  This temporary change in notation is to prevent any ambiguity that might arise since 
  we regard $f_s^{\mathrm{an}}$ as the analytification of the morphism 
  $f_s : X_s \to Y_s$. 
   We then have that  
  for every $x_1,x_2 \in X_s(k)$,
if  $h_s'^{\mathrm{an}}(0,x_1) = h_s'^{\mathrm{an}}(0,x_2)$ then 
 $\mathrm{prfl}(f_s^{\mathrm{an}},x_1) = \mathrm{prfl}(f_s^{\mathrm{an}},x_2)$.
 \end{enumerate}  
  Observe that (2) above allows us to extend the function $\mathrm{prfl}(f_s^{\mathrm{an}},-)$ 
  to the skeleton $\Sigma'_s$. Indeed, given $p \in \Sigma'_s$, 
  we set $\mathrm{prfl}(f_s^{\mathrm{an}},-)$ to be the function 
  $\mathrm{prfl}(f_s^{\mathrm{an}},x)$ for any $x$ that retracts to $p$. 
  By (2), $\mathrm{prfl}(f_s^{\mathrm{an}},p)$ is well defined and we have a function 
  $$\mathrm{prfl}(f_s^{\mathrm{an}},-) : \Sigma'_s \to \mathrm{Fn}((0,\infty],(0,\infty])).$$
  
  The principal assertion of Theorem \ref{tameness of profile functions in families} concerns 
  a finiteness property for the family 
  $$\{f_s^{\mathrm{an}} : \Sigma'_s \to \Sigma_s, \mathrm{prfl}(f_s^{\mathrm{an}},-)_{|\Sigma'_s}\}_{s \in S}.$$ 
More precisely, we show that
  there exists a deformation retraction 
  $g^{\mathrm{an}} : I \times S^{\mathrm{an}} \to S^{\mathrm{an}}$ whose image $\Sigma(S) \subset S^{\mathrm{an}}$ 
  is homeomorphic to a finite simplicial complex and is such that the following holds. 
   The tuple $(f_s^{\mathrm{an}} : \Sigma'_s \to \Sigma_s,\mathrm{prfl}(f_s^{\mathrm{an}},-))$ is constant along the fibres of the retraction of the deformation 
  $g^{\mathrm{an}}$ i.e. if $e$ denotes the end point of the interval $I$ and 
  if $s_1,s_2 \in S(k)$ are such that $g^{\mathrm{an}}(e,s_1) = g^{\mathrm{an}}(e,s_2)$ then 
  $\Sigma'_{s_1} = \Sigma'_{s_1}$ and $\Sigma_{s_1} = \Sigma_{s_1}$. Furthermore, with these identifications 
  $(f_{s_1}^{\mathrm{an}})_{|\Sigma'_{s_1}} = (f_{s_2}^{\mathrm{an}})_{|\Sigma'_{s_2}}$ and 
 $\mathrm{prfl}(f_{s_1}^{\mathrm{an}},-)_{\Sigma'_{s_1}} = \mathrm{prfl}(f_{s_2}^{\mathrm{an}},-)_{|\Sigma'_{s_2}}$.

               This theorem represents one of the primary reasons behind introducing the Backward branching
       indices as they allow us to use the Hrushovski-Loeser construction to control the combinatorial data 
       as it varies in families. 
       It should be noted that constructing deformations in higher dimensions
       is in general very difficult and the Hrushovski-Loeser construction is the only tool
       that allows for some flexibility.

\subsection{Notation} : We fix the field $k$ which is algebraically closed complete with respect to a non-Archimedean non-trivial real valuation. 
By a $k$-curve, we mean a connected variety of dimension $1$.   
\\

\noindent \textbf{Acknowledgements:} We are grateful to Jérôme Poineau for inviting us to the \emph{Berkovich spaces, Tropical geometry and Model theory} summer school
 held in Bogota in July 2017 where he asked us a question concerning the structure of definable subsets of $\widehat{C}$. This question inspired some of the main results of our paper.
 We are also grateful to the other participants at the school and most importantly Pablo Cubides for organising and the discussions. We are indebted to François Loeser, 
 Michael Temkin and Antoine Ducros
 for the discussions at the Oberwolfach workshop in December, 2015. This article would not have been possible without the support and mentorship of Johannes Nicaise 
 and Imperial College, London. 
 Lastly, we thank Andreas Gross for his support and always being available to answer questions. During the period in which this article was written, we were
  supported by the ERC Starting Grant MOTZETA (project 306610) of the European Research Council (PI: Johannes Nicaise).

\section{Definable subsets of $\widehat{C}$} \label{definable subsets of C}
  
      The theory ACVF is an $\mathcal{L}_{k\Gamma}$-theory where 
      $\mathcal{L}_{k\Gamma}$ is the three sorted language 
      consisting of the valued field sort VF, value group sort $\Gamma$ and 
      residue field sort $k$. We add the functions, relations and constants 
      necessary to make any structure of $\mathcal{L}_{k\Gamma}$ 
      a valued field. 
       Within this framework, the theory ACVF consists of those sentences (formulae without quantifiers)
      such that any model of ACVF is an algebraically closed field. 
      Although $\mathcal{L}_{k\Gamma}$ is easily defined, we 
      note that it does not eliminate imaginaries and hence we work in the extended 
      language $\mathcal{L}_{\mathcal{G}}$ (cf. \cite[\S 2.7]{HL}) where we add the geometric sorts  
      to $\mathcal{L}_{k\Gamma}$. For simplicity, we write $\mathcal{L} := \mathcal{L}_{\mathcal{G}}$. 
      For an introduction to Model theory, we refer the reader to \cite[\S 1]{M02}.  
   
   \subsection{Definable sets}   
         We fix a large saturated model $\mathbb{U}$ of ACVF as in \cite[\S 2.1]{HL} 
                              We introduce the notion of a definable set as in \cite{HL} which is 
                              of a more geometric flavour than that 
              presented in \cite{M02}. 
                Let $C \subset \mathbb{U}$ be a parameter set. 
                Let $\phi$ be an $\mathcal{L}$-formula. The formula $\phi$ can be used to define a
                 functor $Z_\phi$ from the category whose objects are $\mathcal{L}_C$-models \footnote{Here $\mathcal{L}_C$ 
                 denotes the extension of $\mathcal{L}$ obtained by adding constants corresponding to the parameter set $C$.}
                 of 
                     $\mathrm{ACVF}$ and morphisms
                      are elementary embeddings to the category of sets. Let 
                      $\{x_1,\ldots,x_{m}\}$ be the set of variables which occur in $\phi$.
                      Given an $\mathfrak{L}_C$-model ${M}$ of ACVF i.e. a model of ACVF that contains $C$, we set 
                     \begin{align*}
                        Z_{\phi}({M}) = \{\bar{a} \in  M^m | {M} \models \phi(\bar{a}) \} .  
                     \end{align*}  
                           Clearly, $Z_{\phi}$ is well defined. 
                  
                \begin{defi} \emph{
                      A} $C$-definable set  $Z$
                     \emph{is a functor from the category whose objects are $\mathcal{L}_C$-models of 
                     ACVF and morphisms are elementary embeddings to the category of sets such that there exists an $\mathfrak{L}_C$-formula 
                     $\phi$ and $Z = Z_{\phi}$. } 
                \end{defi}  
                
       \begin{es} 
      \emph{  Consider the following examples of three different classes of definable sets in ACVF. Observe that these objects 
        appear naturally in the study of Berkovich spaces, tropical geometry and algebraic geometry. 
      \begin{enumerate}
      \item  Let $K$ be a model of ACVF and $A \subset K$ be a set of parameters. Let $n \in \mathbb{N}$. 
       A semi-algebraic subset of $K^n$ is a finite boolean combination of sets of the form 
       $$\{x \in K^n | \mathrm{val}(f(x)) \geq \mathrm{val}(g(x))\}$$
       where $f,g$ are polynomials with coefficients in $A$. One verifies that semi-algebraic sets 
       extend naturally to define $A$-definable sets in ACVF. 
      \item Let $G = \Gamma(K)$ where $K$ is as above. A $G$-rational polyhedron 
      is a finite Boolean combination of subsets of the form 
      $$\{(a_1,\ldots,a_n) \in G^n | \Sigma_i z_ia_i \leq c \}$$
      where $a_i \in \mathbb{Z}$ and $c \in G$. Such objects extend
      naturally to define a $K$-definable subset of $\Gamma^n$ where $\Gamma$ is the value group sort. 
      \item  Any constructible subset of $k^n$ gives a definable set in a natural way. 
       \item Let $C$ be a $k$-curve. Then the space of stably dominated types - $\widehat{C}$ is a $k$-definable set. 
       This is unique to the one dimensional case i.e. the space 
       $\widehat{V}$ where $V$ is a $k$-variety of dimension atleast $2$ 
       is \emph{strictly pro-definable}. 
       \end{enumerate} }
       \end{es}         
%
\subsection{Homotopies on $\widehat{C}$.} \label{homotopies of C}
   
        Let $C$ be a projective $k$-curve. The space $\widehat{C}$ is Hausdorff and 
       definably path connected. Furthermore, it admits a definable deformation 
        retraction 
        $$H : [0,\infty] \times \widehat{C} \to \widehat{C}$$ such that the image 
        of $H$ is a $k$-definable $\Gamma$-internal set i.e. a $k$-definable subset 
        of $\widehat{C}$ that is in $k$-definable bijection with a $k$-definable 
      subset of $\Gamma^n$ for some $n \in \mathbb{N}$. 
      We briefly explain the construction of the homotopy $H$ from \cite[\S 7.5]{HL} 
      as this will play a crucial role in the following sections. 
      
          We fix a system of coordinates on $\mathbb{P}^1$. 
          Let $m : \mathbb{P}^1 \times \mathbb{P}^1 \to \Gamma_\infty$ 
          be the standard metric as in \cite[\S 3.10]{HL}. 
         Let $\mathcal{O}_0 := \{x \in \mathbb{P}^1 | m(x,0) \geq 0\}$ and 
         $\mathcal{O}_\infty := \{x \in \mathbb{P}^1 | m(x,\infty) \geq 0\}$. 
         Observe that $\mathcal{O}_0$ and $\mathcal{O}_\infty$ are
         definable subsets of $k$ and are the closed unit disk around $0$ and 
         $\infty$ respectively. 
       Using $m$, we define a $v+g$-continuous homotopy, 
       $\psi : [0,\infty] \times \mathbb{P}^1 \to \widehat{\mathbb{P}^1}$    
    by mapping $(t,a)$ to the generic type \cite[Example 3.2.1]{HL} of the 
    valuative ball 
    $$\{x \in \mathbb{P}^1| m(x,a) \geq t\}.$$ 
    The map $\psi$ extends to 
    define a map 
    $\widehat{\psi} : [0,\infty] \times \widehat{\mathbb{P}^1} \to \widehat{\mathbb{P}^1}$
    whose image is the generic type of the ball $\mathcal{O}_0$ and is hence $\Gamma$-internal. 
    
     Let $D \subset \mathbb{P}^1(k)$ be a finite set. 
     We define ${m}_D : \mathbb{P}^1 \to \Gamma_\infty$ 
     by $x \mapsto \mathrm{max}\{m(x,d) | d \in D\}$. 
     We can extend the homotopy 
     $\psi$ above to a homotopy 
     \begin{align*} 
     \psi_D : [0,\infty] &\times \mathbb{P}^1 \to \widehat{\mathbb{P}^1}  \\
     &(t,a) \mapsto \psi(\mathrm{max}\{t,{m}_D(a)\},a)
     \end{align*}
     As before, $\psi_D$ extends to 
     a homotopy 
     $\widehat{\psi_D} : [0,\infty] \times \widehat{\mathbb{P}^1} \to \widehat{\mathbb{P}^1}$
     whose
     image $\Upsilon_D$ is the convex hull of $D$ in $\widehat{\mathbb{P}^1}$.
     We refer to $\psi_D$ as the \emph{cut-off} of $\psi$ with respect to the function 
     $m_D$ (cf. \cite[Lemma 10.1.6]{HL}).
%

       A homotopy $h : [0,\infty] \times C \to \widehat{C}$ is constructed 
       as follows. 
       Let $f : C \to \mathbb{P}^1$ be a finite morphism. 
        By \cite[\S 7.5.1]{HL}, there exists a divisor $D \subset \mathbb{P}^1$ 
        such that 
        ${\psi_D}$ lifts \emph{uniquely} to a 
        definable map 
        $h : [0,\infty] \times C \to \widehat{C}$. 
        This means that there exists a unique
        $h : [0,\infty] \times C \to \widehat{C}$ such that 
        the diagram below commutes. 

$$
 \begin{tikzcd} [row sep = large, column sep = large] 
I \times C \arrow[r,"H"] \arrow[d, "id \times f"] &
\widehat{C}  \arrow[d, "\widehat{f}"] \\
I  \times \mathbb{P}^1 \arrow[r,"{\psi_D}"] &
\widehat{\mathbb{P}^1} 
\end{tikzcd}
$$        
    where $I := [0,\infty]$. 
    By \cite[Lemma 10.1.1]{HL}, 
    the uniqueness of the map $h$ guarantees that the canonical extension (cf. \S 3.8 in loc.cit.) 
    $H : [0,\infty] \times \widehat{C} \to \widehat{C}$ 
    is a homotopy whose image is $\widehat{f}^{-1}(\Upsilon_D)$ and is $\Gamma$-internal. 
   Note that we often say a map 
   $g : [0,\infty] \times C \to \widehat{C}$ is a homotopy.
   By this we mean that it is $v+g$-continuous and hence its canonical extension 
   $G : [0,\infty] \times \widehat{C} \to \widehat{C}$ is a homotopy.

   \begin{rem} \label{metric deformations}
    \emph{We make the following observation that is crucial to relating the 
    deformation retractions above to the metric structure of the Berkovich curve $C^{\mathrm{an}}$. 
    Recall from \S 1, the field $k^{max}$ which is a 
    maximally complete algebraically closed extension of $k$ whose value group is $\mathbb{R}^*$.
    We have a morphism 
    $\pi_{k,C} : \widehat{C}(k^{max}) \to C^{\mathrm{an}}$ that factors through the canonical
    homeomorphism $\pi : \widehat{C}(k^{max}) \to C_{k^{max}}^{\mathrm{an}}$.}
    
    \emph{The map $\pi_{k,\mathbb{P}^1}$ is such that if $a \in \mathbb{P}^1(k)$ and 
    $\gamma \in \Gamma(k)$ then the generic type of the ball around 
    $a$ of valuative radius $\gamma$ belongs to $\widehat{\mathbb{P}^1}(k^{max})$
    and $\pi_{k,\mathbb{P}^1}$ maps this point to the point $\zeta_{a,-\mathrm{log}(\gamma)}$.} 
     
    \emph{Let $D \subset \mathbb{P}^1(k)$ be a divisor. 
    Let  $x \in \widehat{\mathbb{P}^1}(k) \subset \widehat{\mathbb{P}^1}(k^{max})$. 
   Let $t_1,t_2 \in [0,\infty](\mathbb{R})$ with $t_1 < t_2$ and 
     such that the map $t \mapsto \widehat{\psi_D}(t,x)$ is injective when restricted 
     to $[t_1,t_2]$. Let $\rho$ denote the metric on 
   the space $\mathbb{P}^{1,\mathrm{an}}$ which we refer to as its skeletal metric or 
   path distance metric (cf. \cite[\S 5]{BPR}). 
   We then have that $$t_2 - t_1 = \rho(\pi_{k,\mathbb{P}^1}\psi_D(t_1,x),\pi_{k,\mathbb{P}^1}\psi_D(t_2,x)).$$}
   \end{rem}  
    
   \subsection{Radiality} \label{radiality}
      
            One of our goals is to understand the topological nature of definable subsets. 
            In general, a definable subset of $\widehat{C}$ does not necessarily have finitely many 
           definably path connected components. For example, consider $C \subset \widehat{C}$. 
            Hence we restrict our attention to definably path connected definable sets
            and use in a novel way the homotopies discussed in \S\ref{homotopies of C} 
             to study them. 
 
         Let $C$ be a projective $k$-curve. 
           Let $H : [0,\infty] \times \widehat{C} \to \widehat{C}$ be a homotopy 
        of $\widehat{C}$ with image $\Upsilon$ as constructed in
         \S \ref{homotopies of C}. 
        Let $\mathfrak{A}(H)$ denote the set of 
     definably path connected definable subsets of $\widehat{C}$
     that intersect $\Upsilon$ in atleast one point. 
   
    Let $\gamma_1 < \gamma_2 < 0$ be elements in $[0,\infty]$. 
     Let $\mathfrak{B}(H)$ denote the set of 
       definable functions (cf. \cite[\S 3.7]{HL})
      $\alpha : C \to [0,\infty] \times \{\gamma_1,\gamma_2\}$
     such that the following properties hold.
    \begin{enumerate}
    \item  For every $x,y \in C$
     if 
     $$t_{x,y} := \mathrm{sup}\{t \in [0,\infty] | H(t,x) = H(t,y)\}$$ 
     then 
     \begin{align*} 
    t_{x,y} \leq & p_1(\alpha(x)) \implies t_{x,y} \leq p_1(\alpha(y)) \\ 
    &\mbox{and} \\ 
     t_{x,y}  > & p_1(\alpha(x)) \implies \alpha(x) = \alpha(y)
    \end{align*} 
    where $p_1$ is the projection onto the first coordinate.
    \item Given $x \in C(k)$, let $e(x)$ be the supremum of the set of 
    elements $t \in [0,\infty]$ such that $H(-,x)$ is constant on the interval $[0,t]$.
    We have that for every $x \in C(k)$, $\alpha(x) \in [e(x),\infty]$. 
    \item The complement in $\Upsilon$ 
    of the space $\{H(0,x) | \alpha(x) = (e(x),\gamma_2)\} \subset \Upsilon$ 
     is non-empty and definably path connected. 
    \end{enumerate}
            The properties required of the elements in 
      $\mathfrak{B}(H)$ is motivated by the following proposition. 
      
            \begin{prop} \label{classification of definable sets} 
       Let $C$ be a projective $k$-curve. 
 Let $H : [0,\infty] \times \widehat{C} \to \widehat{C}$ be a homotopy 
        of $\widehat{C}$ with image $\Upsilon$ as constructed in
        Section \S \ref{homotopies of C}.
                      We then have a bijection 
             $\beta : \mathfrak{A}(H) \to \mathfrak{B}(H)$.   
       \end{prop} 
       \begin{proof} 
      Let $X \subset \widehat{C}$ 
       be a
       definably path connected definable subset. 
       We define a
       function $\beta_{X,H} : C \to [0,\infty]$
       as follows. 
       Let $a \in C$. 
       The homotopy 
       ${H}$ defines 
       a  path $[0,\infty] \to \widehat{C}$ given 
       by $t \mapsto H(t,a)$. 
       Suppose that 
       $H(t,a) \in X$ for some $t \in [0,\infty]$. 
        Let 
       $$\beta'_{X,H}(a) := \mathrm{sup}\{t \in [0,\infty] | H(t,a) \in X\}.$$
       If $H(\beta'_{X,H}(a),a)$ belongs 
       to $X$ then 
       we set 
      $$\beta''_{X,H}(a) := \gamma_1$$ and 
      if $H(\beta'_{X,H}(a),a)$ does not belong 
       to $X$ then 
       we set 
      $$\beta''_{X,H}(a) := \gamma_2.$$
      Let $\beta_{X,H}(a) := (\beta'_{X,H}(a),\beta''_{X,H}(a))$. 
      Suppose on the other hand that for every $t \in [0,\infty]$,
       $H(t,a) \notin X$. 
       We then define 
       $\beta_{X,H}(a) := (e(a),\gamma_2)$ where the notation is as introduced 
       above. 
      
      Clearly, 
      $\beta_{X,H}$ defines a map $C \to [0,\infty] \times \{\gamma_1,\gamma_2\}$
     that is 
      definable and 
      one checks that 
      $\beta_{X,H} \in \mathfrak{B}(H)$.
       We verify the surjectivity of the map $X \mapsto \beta_{X,H}$ as follows.  
       Suppose 
       $\alpha : C \to [0,\infty]$ belonged to the family 
       $\mathfrak{B}(H)$. 
       We define  
       $X_\alpha$ to be the set of $z \in \widehat{C}$ such that there 
       exists $x \in C$, $t \in [0,\alpha(x)]$ and
       $z = H(t,x)$.  
        Clearly, $X_{\alpha}$ is a definable set since 
        $C$ is definable and $\alpha$ is a definable function.
      Let $S' \subset \Upsilon$ be the set 
     $\{H(0,x) | \alpha(x) = (e(x),\gamma_2)\}$ 
     and let $S$ denote its complement in $\Upsilon$. 
     By condition (2) above, we see that 
     $S$ is definably path connected and non-empty.
     We deduce that $X_\alpha$ intersects $\Upsilon$ non-trivially
     and must be definably path connected. 
     One verifies directly that 
     if $X \in \mathfrak{A}(H)$ then
     $X_{\beta_{X,H}} = X$ and 
     if $\alpha \in \mathfrak{B}(H)$ then 
     $\beta_{X_\alpha,H} = \alpha$. 
              \end{proof}

     \begin{thm} \label{radiality of definable sets}
     Let $C$ be a projective $k$-curve and
     let $f : C \to \mathbb{P}^1$ be a finite morphism. 
     Let $X \subset \widehat{C}$ 
       be a definably path connected definable subset. 
     There exists a 
     homotopy 
     $H_1 : [0,\infty] \times \widehat{C} \to \widehat{C}$
     whose image $\Upsilon$ is a 
       $\Gamma$-internal subset of $\widehat{C}$ 
       such that $X \in \mathfrak{A}(H_1)$ and the 
       function $\beta_{X,H_1}$ (as in Proposition \ref{classification of definable sets})
        is constant along the fibres of the 
       retraction. Furthermore, $H_1$ is the unique lift of a homotopy 
       $\widehat{\psi_{D_1}}$ where $D_1 \subset C$ is a Zariski closed subset of $C$. 
       \end{thm} 
    \begin{proof} 
      Let $h : [0,\infty] \times C \to \widehat{C}$ be a homotopy that is the unique 
      lift of a homotopy 
      $\psi_D : [0,\infty] \times \mathbb{P}^1 \to \widehat{\mathbb{P}^1}$ where 
      $D \subset \mathbb{P}^1$ is a divisor. 
      We use $H$ to denote the canonical extension 
      $[0,\infty] \times \widehat{C} \to \widehat{C}$ of $h$. 
      Let $\xi_1 := p_1 \circ \beta_{X,H}$ and $\xi_2 := p_2 \circ \beta_{X,H}$
       where $p_i : \Gamma_\infty^2 \to \Gamma_\infty$ is the projection map 
       to the $i$-th coordinate. 
        By \cite[Lemma 10.2.3]{HL}, there exists a finite family 
        $\{\xi'_1,\ldots,\xi'_r\}$ of definable functions 
        $\mathbb{P}^1 \to \Gamma_\infty$ and a divisor $D_0 \subset \mathbb{P}^1$ 
        such that  
        any homotopy on $C$ that is the lift of a homotopy on $\mathbb{P}^1$
        that preserves the functions \footnote{We say a homotopy $h$ preserves a function $\xi$ if the function 
        $\xi$ is constant for the retraction map.} $\xi'_i$ and 
        fixes the divisor $D_0$ must preserve the functions $\xi_i$. 
        Since the functions $\xi'_i$ are definable, we can enlarge the divisor $D_0$ 
        so that for every $i$, there exists a definable partition $X_{i1},\ldots,X_{ir_i}$ of 
        $U_0 := \mathbb{P}^1 \smallsetminus D_0$,
        $(\xi'_i)_{|X_{it}}$ is of the form $\mathrm{val}(g_{it})$ where 
        $g_{it}$ is regular on $U_0$ and the set $X_{it}$ is given by the boolean combination 
        of sets of the form 
        $\{x \in U_0 | \mathrm{val}(a_{it})(x) \geq \mathrm{val}(b_{it})(x)\}$ 
        where $a_{it}$ and $b_{it}$ are regular on $U_0$.
        Let $\{c_1,\ldots,c_s\}$ be the finite family of regular functions on $U_0$ 
        consisting of the $g_{it}$'s and $a_{it}$'s and $b_{it}$'s as above. 
        It can be verified that any homotopy of $U_0$ that preserves the functions 
        $\mathrm{val}(c_j)$ preserves $\xi'_i$ for all $i$. 
        Lastly, 
        let $D_1$ be the union of $D_0$, $D$ and the zeros of the regular  
        functions $c_i$.
        We have that $\psi_{D_1}$ preserves the functions $c_i$. 
        Furthermore, $\psi_{D_1}$ is the cut-off $\psi_D[m_{D_1}]$ 
        where $m_{D_1}$ is as defined in \S \ref{homotopies of C}.
        Hence, we see that the cut-off $h[m_{D_1} \circ f]$ is the unique lift 
        of $\psi_{D_1} = \psi_D[m_{D_1}]$.
        
           We simplify notation and write $h_1 := h[m_{D_1} \circ f]$ 
           and let $H_1$ be its canonical extension.
           Let $\Upsilon_1$ denote the image of the deformation $H_1$ i.e. 
           $\{H_1(0,x) | x \in \widehat{C}\}$. 
           By construction, the function $\beta_{X,H}$ 
           is constant along the fibres of the retraction 
           $H_1(0,-) : \widehat{C} \to \Upsilon_1$. 
           We check that this implies $\beta_{X,H_1}$ is also
           constant along the fibres of the retraction.
           Let $x \in \Upsilon_1$. 
        Let $y_1,y_2 \in C$ such that $H_1(0,y_1) = H_1(0,y_2) = x$. 
        We must show that $\beta_{X,H_1}(y_1) = \beta_{X,H_1}(y_2)$. 
        Let $t_{y_1} := \mathrm{sup}\{t |H_1(t,y_1) \in \Upsilon_1\}$. 
         We have that $t_{y_1} = t_{y_2}$. 
         It suffices to verify this for the homotopy $\psi_{D_1}$
        and in this case it is a consequence of the fact that $m_{D_1}$ is 
        constant along the fibres of the retraction.
        Suppose, 
        $p_1(\beta_{X,H}(y_1)) \leq t_{y_1}$ then 
        $\beta_{X,H}(y_1) = \beta_{X,H_1}(y_1)$ and 
        likewise, $\beta_{X,H}(y_2) = \beta_{X,H_1}(y_2)$.
         Since $\beta_{X,H}$ is constant along the fibres of the retraction
         of $H_1$, we verify the claim in this case. 
        If 
        $p_1(\beta_{X,H}(y_1)) > t_{y_1}$
        then we get that 
        $\beta_{X,H_1}(y_1) = \beta_{X,H_1}(y_2) = (t_{y_1},\gamma_2)$ (or $= (t_{y_1},\gamma_1)$).  
         \end{proof} 
              
        \begin{rem} \label{extensions preserve radiality}
            \emph{Let $C$ be a $k$-curve and
     let $X \subset \widehat{C}$ 
       be a definably path connected definable subset. 
     Let 
     $H : [0,\infty] \times \widehat{C} \to \widehat{C}$ be a homotopy
     whose image $\Upsilon$ is a 
       $\Gamma$-internal subset of $\widehat{C}$ 
       such that $X \in \mathfrak{A}(H)$ and the 
       function $\beta_{X,H}$ (as in Proposition \ref{classification of definable sets})
        is constant along the fibres of the 
       retraction. Let us assume that $H$ is the canonical extension of a homotopy 
       $h : [0,\infty] \times C \to \widehat{C}$. 
       Let $r : C \to [0,\infty]$ be a $v+g$-continuous function such that 
       the image of the cut-off homotopy 
       $h[r]$ is $\Gamma$-internal as well. 
       The arguments in the proof 
       of Theorem \ref{radiality of definable sets} 
      shows that
       $\beta_{X,H[r]}$ is constant along the fibres of the retraction where 
       $H[r]$ is the cut-off of $H$ via the function $r$. It is also the 
       canonical extension of $h[r]$. }

        \end{rem}       
              
   \subsection{Backward branching index} \label{backward branching index}    
     
         The homotopies $\psi_D$ on $\mathbb{P}^1$ respect the metric structure of
         the curve $\mathbb{P}_k^{1,\mathrm{an}}$ as
         observed in Remark \ref{metric deformations}. 
         However, lifts
         $$h : [0,\infty] \times C \to \widehat{C}$$ 
                   of $\psi_D$ for a morphism $f : C \to \mathbb{P}^1$
             no longer satisfy this property. 
        Associated to $h$ and the morphism $f$, for every $x \in C$, 
        we introduce definable functions 
        $$BB_{h,f}(x) : [0,\infty] \to \mathbb{N} \times [0,\infty]$$ 
        \footnote{If $X$ is a definable set then a definable function $a : X \to \mathbb{N}$
         is a map with finite image and for every $n \in \mathbb{N}$, $a^{-1}(n)$ is a definable subset of $X$.}
         which will later help reconcile
        the homotopy $H$ with the
         metric properties of 
        $C^{\mathrm{an}}$.
       
     We define the Backward branching index function more generally 
        for any finite morphism $g : C' \to C$ of smooth projective irreducible curves. 
       Let $h : [0,\infty] \times C \to \widehat{C}$ be a deformation retraction
        of $C$ onto a $\Gamma$-internal set $\Sigma$. 
        Let $h' : [0,\infty] \times C' \to \widehat{C'}$ be a lift 
        of $h$ and let $\Sigma'$ denote the
        $\Gamma$-internal set which is the
         image of $h'$. 
         We will assume that $h$ is obtained as in 
         \S \ref{homotopies of C} by lifting a homotopy 
         $\psi_D$ on $\mathbb{P}^1$.
          
          \begin{defi} \label{BB index}
          The backward branching index of a point $x \in C'$ \emph{is a function}
           $$BB_{h',g}(x) : [0,\infty] \to \mathbb{N} \times [0,\infty]$$
         \emph{defined as follows. 
         If $t \in [0,\infty]$, let $B_{h',g}(x)(t)$ denote the set of 
          elements $z \in C'$ such that 
        $g(z) = g(x)$, 
        $h'(t,z) = h'(t,x)$ and $h'(t,z) \neq h'(0,z)$.
        Then 
        \begin{align*}
        BB_{h',g}(x)(t) := (\mathrm{card}(B_{h',g}(x)(t)), \Sigma_{z \in B_{h',g}(x)(t)} \mathrm{val}(z)).
        \end{align*}
        By Remark \ref{definability of backward branching index}, 
        we
        thus have a function 
        $$BB_{h',g} : C' \to \mathrm{Fn}([0,\infty],\mathbb{N} \times [0,\infty])$$
         where  $\mathrm{Fn}([0,\infty],\mathbb{N} \times [0,\infty])$ is the set of definable
        functions $[0,\infty] \to \mathbb{N} \times [0,\infty]$.}
        \end{defi} 
        \begin{rem} \label{definability of backward branching index}
    \emph{The finiteness of $g$ implies that 
        the $BB_{h',g}(x)$ takes on 
        only finitely many values. 
        It follows that there exists $n_x \in \mathbb{N}$ such that 
        the tuple $$S_x := ((s_0,e_0,\beta_0),\ldots,(s_{n_x},e_{n_x},\beta_{n_x})) \subset ([0,\infty] \times \mathbb{N} \times [0,\infty])^{n_x}$$ 
        with $s_0 = \infty$, completely determines 
        $BB_{h',g}(x)$. 
        By this we mean that if 
        $s \in [s_i,s_{i+1})$ then 
        $BB_{h',g}(x)(s) = (e_i,\beta_i)$ and 
        if $s \in [s_{n_x},0]$ then 
        $BB_{h',g}(x)(s) = (e_{n_x},\beta_{n_x})$
          The $s_i$ are the break points of the function i.e. $e_i \neq e_{i+1}$.}
          \end{rem} 
         
 \begin{thm} \label{backward branching is constant along retractions}
 Let $g : C' \to C$ be a finite morphism of smooth projective irreducible curves. 
       Let $h : [0,\infty] \times C \to \widehat{C}$ be a deformation retraction
        of $C$ onto a $\Gamma$-internal set $\Sigma$. 
        Let $h' : [0,\infty] \times C' \to \widehat{C'}$ be a lift 
        of $h$ and let $\Sigma'$ denote the
        $\Gamma$-internal set which is the
         image of $h'$. 
         We will assume that $h$ is obtained as in 
         \S \ref{homotopies of C} by lifting a homotopy 
         $\psi_D$ on $\mathbb{P}^1$.
There exists a $v+g$-continuous function 
$\alpha : C' \to \Gamma_\infty$ such that
if $h'_1$ denotes the cut-off 
$h'[\alpha]$ then 
the function $BB_{h'_1,g}$ is constant 
 along the fibres of the retraction associated to $h'_1$.
  \end{thm}         
 \begin{proof} We begin by showing that 
       $h'_1$ can be chosen so that 
       $BB_{h',g}$ is constant    
       along the retraction $h'_1(0,-)$.
   Observe that 
   $h'$ is a lift of a homotopy $\psi_D$ on $\mathbb{P}^1$.
    Using the arguments in the proof of
    Theorem \ref{radiality of definable sets}, 
     we observe that it suffices to show that there 
     exists finitely many definable functions 
     $\{\xi_1,\ldots,\xi_m\}$ such that 
     if there exists a $v+g$-continuous function 
     $\alpha : C' \to \Gamma_\infty$ and $\xi_i$ are constant along the 
     retraction associated to
     $h'[\alpha]$ then $h'[\alpha]$ preserves 
     $BB_{h',g}$.
         Let $N \in \mathbb{N}$ be the supremum of the 
         set $\{\mathrm{card}\{g^{-1}(y)\} | y \in C\}$. 
         By definition,
         $n_x \leq N$ for every $x \in C$. 
     Recall that if $x \in C$ then
         $$S_x := ((s_0,e_0,\beta_0),\ldots,(s_{n_x},e_{n_x},\beta_{n_x}))$$ is the tuple 
         from Remark \ref{definability of backward branching index}
         that determines $BB_{h',g}(x)$. 
         Observe that the $s_i$ are distinct and correspond to the set of
         $t \in [0,\infty]$ such that there exists $z \in C$, $h(t,x) = h(t,z)$,
         $g(t) = g(z)$ and $h(t,z) \neq h(0,z)$. 
        It follows that $\{s_1,\ldots,s_{n_{x}}\}_{x \in C}$ varies
         definably along $C$
         and hence $x \mapsto n_x$ is a definable function. 
                  Let $\gamma_1,\ldots,\gamma_N$ be $N$ distinct points in $\Gamma(k)$. 
      We then have that the function $\xi$ given by $x \mapsto \gamma_{n_x}$ is 
         definable.  
       
         
         Once again let 
         $$S_x := ((s_0,e_0,\beta_0),\ldots,(s_{n_x},e_{n_x},\beta_{n_x}))$$
         be the tuple
          that 
         determines $BB_{h',g}(x)$.
         Let $\gamma < 0 \in \Gamma(k)$. 
              For every 
     $1 \leq i \leq N$, we define 
     a family $\{\xi_{i1},\ldots,\xi_{ii}\}$ 
    where for every 
    $x \in \xi^{-1}(\gamma_i)$, 
    $\xi_{ij}(x) := s_j$ 
     and if $x \notin \xi^{-1}(\gamma_i)$ then 
    $\xi_{ij}(x) := \gamma$. Since $\{BB_{h',g}(x)\}_{x \in C'}$ is uniformly definable in $x$, we get that 
    $\{S_x\}_{x \in C'}$ is also uniformly definable and hence 
    the $\xi_{ij}$ are definable functions. 
    
         We now encode the integers $e_i$ that appear in the tuple $S_x$ 
         using definable functions. 
         As before, 
          we see 
         that 
         for every 
         $x \in C'$ if 
         $S_x = ((s_0,e_0,\beta_0),\ldots,(s_{n_x},e_{n_x},\beta_{n_x}))$ then for every $l$, 
         $e_l \leq N$. 
         Let $\theta_1 < \ldots < \theta_N$ be elements of $\Gamma$. 
         For every $1 \leq i \leq N$, we define a family 
         $\{\eta_{i1},\ldots,\eta_{ii}\}$ of functions from $C'$ to $\Gamma_\infty$
         as follows.
         If $x \in \xi^{-1}(\gamma_i)$ then we set 
         $\eta_{ij}(x) := \theta_{e_j}$ and 
         if $x \notin \xi^{-1}(\gamma_i)$ then we set 
         $\eta_{ij}(x) := \gamma$. 
         
         Similarly, for every 
     $1 \leq i \leq N$, we define 
     a family $\{\kappa_{i1},\ldots,\kappa_{ii}\}$ 
    where for every 
    $x \in \xi^{-1}(\gamma_i)$, 
    $\kappa_{ij}(x) := \beta_j$ 
     and if $x \notin \xi^{-1}(\gamma_i)$ then 
    $\kappa_{ij}(x) := \gamma$. As before, the definability of 
    $\kappa_{ij}$ follows from the fact $\{S_x\}_{x \in C'}$ is uniformly definable.

         Clearly, any homotopy that fixes 
         the definable functions $\xi_{ij}$, $\eta_{ij}$ and 
         $\xi$ must preserve $BB_{h',g}$. 
         Hence we see that 
         there exists a $v+g$-continuous function 
     $\alpha : C' \to \Gamma_\infty$ such that the 
     retraction associated to
       $h'_1 := h'[\alpha]$ preserves 
     $BB_{h',g}$.
          
          Let $x \in \Sigma_1$. 
        Let $y_1,y_2 \in C$ such that $h'_1(0,y_1) = h'_1(0,y_2) = x$. 
        We must show that $BB_{h'_1,g}(y_1) = BB_{h'_1,g}(y_2)$.   
        Let $t_{y_1} := \mathrm{sup}\{t |h'_1(t,y_1) \in \Sigma_1\}$. 
       As in Lemma \ref{radiality of definable sets}, we have that $t_{y_1} = t_{y_2}$. 
       Since $h'_1$ is a cut-off of $h'$,
      for every $t > t_{y_1}$, 
      $BB_{h'_1,g}(y_1)(t) = BB_{h',g}(y_1)(t)$ and 
      by construction of $h'_1$, we get 
        that $BB_{h',g}(y_1)(t) = BB_{h',g}(y_2)(t)$ for $t \geq t_{y_1} = t_{y_2}$.
        This shows that 
        $BB_{h'_1,g}(y_1)_{|(t_{y_1},\infty]} = BB_{h'_1,g}(y_1)_{|(t_{y_2},\infty]}$.
        Since we can choose $\epsilon > 0$ such that  
        $$BB_{h'_1,g}(y_1)(t_{y_1}) = BB_{h'_1,g}(y_1)(t_{y_1} + \epsilon)$$
        and 
        $$BB_{h'_1,g}(y_2)(t_{y_2}) = BB_{h'_1,g}(y_2)(t_{y_2} + \epsilon),$$ 
        we conclude that 
        $$BB_{h'_1,g}(y_1) = BB_{h'_1,g}(y_1).$$
      \end{proof} 
   
         \begin{rem} \label{treating multiple morphisms}
   \emph{Let $f_1 : C' \to \mathbb{P}^1$ and $f_2 : C \to \mathbb{P}^1$ be two finite morphisms between 
   smooth projective irreducible $k$-curves. 
   Let $h_1 : [0,\infty] \times C' \to \widehat{C'}$ and 
   $h_2 : [0,\infty] \times C \to \widehat{C}$ be two homotopies whose images 
   are $\Gamma$-internal and which are the lifts of homotopies
   $\psi_{D_1} : [0,\infty] \times \mathbb{P}^1 \to \widehat{\mathbb{P}^1}$ 
   and $\psi_{D_2} : [0,\infty] \times \mathbb{P}^1 \to \widehat{\mathbb{P}^1}$ respectively. 
   By Theorem \ref{backward branching is constant along retractions}, there exists 
   $v+g$-continuous functions $\alpha_1 : C' \to [0,\infty]$ and 
   $\alpha_2 : C \to [0,\infty]$ such that 
   the cut-offs $h_1[\alpha_1]$ and $h_2[\alpha_2]$ preserve the 
   functions $BB_{h_1,f_1}$ and $BB_{h_2,f_2}$ respectively. 
   The proof of Theorem \ref{backward branching is constant along retractions}
   uses arguments outlined in the proof of Theorem \ref{radiality of definable sets}, 
    from which we see that 
    the functions $\alpha_1$ and $\alpha_2$ are the pull backs of 
   $v+g$-continuous functions 
   $\beta_1 : \mathbb{P}^1 \to [0,\infty]$ and 
   $\beta_2 : \mathbb{P}^1 \to [0,\infty]$ via 
   $f_1$ and $f_2$ respectively. 
   Let 
   $\beta : \mathbb{P}^1 \to [0,\infty]$ be the function 
   given by $x \mapsto \mathrm{max}\{\beta_1(x),\beta_2(x),m_{D_1},m_{D_2}\}$.
   The function $\beta$ is $v+g$-continuous since it can be seen as the composition of $v+g$-continuous functions. 
   One checks using the arguments in the proof of 
   Theorem \ref{backward branching is constant along retractions} that not only do
   $\psi_{D_1}[\beta]$ lift to $h_1[\beta \circ f_1]$ and 
   $\psi_{D_2}[\beta]$ lift to $h_2[\beta \circ f_2]$ but more importantly that
   $h_1[\beta \circ f_1]$ and
   $h_2[\beta \circ f_2]$ respect the functions 
   $BB_{h_1,f_1}$ and $BB_{h_2,f_2}$ 
   respectively.}  
      \end{rem}

\section{Profile functions}

\subsection{Homotopies of $C^{\mathrm{an}}$}  \label{homotopies of C^{an}}
  
     Let $f : C \to \mathbb{P}^1$ be a finite morphism. 
       In \S \ref{definable subsets of C}, we discussed the 
       construction of definable deformation retractions on 
       $\widehat{C}$ via the morphism $f$ 
       and a suitable homotopy 
       $\psi_D : [0,\infty] \times \mathbb{P}^1 \to \widehat{\mathbb{P}^1}$.
       We now interpret these constructions in the analytic setting. 
        
       The homotopy $\psi_D$ induces a homotopy 
       $\psi^{\mathrm{an}}_D : [0,\infty] \times \mathbb{P}^{1,\mathrm{an}} \to \mathbb{P}^{1,\mathrm{an}}$
       whose image is the closed subspace $\Sigma_D := \pi_{k,\mathbb{P}^1}(\Upsilon_D)$ where 
       $\pi_{k,\mathbb{P}^1} : \widehat{\mathbb{P}^1}(k^{max}) \to \mathbb{P}^{1,\mathrm{an}}$ and 
       $\Upsilon_D$ is the convex hull in $\widehat{\mathbb{P}^1}$ of the closed subset $D$. 
       If $D$ is $k$-definable then one checks that for every $x,y \in D$, 
       the path $[x,y]$ in $\widehat{\mathbb{P}^1}$ is $k$-definable. 
       We can describe these paths explicitly and one deduces that
       $\Sigma_D$ is the convex hull in $\mathbb{P}^{1,\mathrm{an}}$ of the closed subset 
       $D$. 
       Recall that $\psi_D = \psi[m_D]$ and we describe $\psi^{\mathrm{an}}$ explicitly as follows. 
       As before we choose a system of coordinates on $\mathbb{P}^1$.  
       The standard metric $m$ on $\mathbb{P}^1$ is defined 
       by identifying $\mathbb{P}^1$ with the glueing of the closed unit ball around $0$ 
       and the closed unit ball around $\infty$ along the 
       annulus $\{x \in \mathbb{P}^1 | T(x) = 1\}$ where $T$ is the 
       coordinate chosen \footnote{See \cite[Proposition 2.7]{WE1} for a generalization of this construction.}. 
       The homotopy $\psi^{\mathrm{an}}$ respects these two copies of the closed disk
       and hence it suffices to describe its restriction to the closed unit disk around $0$.  
           
       Let $B := \{x \in \mathbb{A}^{1,\mathrm{an}}| |T(x)| \leq 1 \}$ be
        the Berkovich closed unit disk around the origin.     
       Let $\zeta_{a,r}$ be a point of type I,II or III with $a \in B(k)$ and $r \in [0,1]$. 
       We have that 
       $\psi^{\mathrm{an}}(t,\zeta_{a,r}) := \zeta_{a,\mathrm{max}\{\mathrm{exp}(-t),r\}}$.
       Since $\psi_D$ is by definition the cut-off $\psi[m_D]$, we deduce that 
       $\psi^{\mathrm{an}}_D = \psi^{\mathrm{an}}[m_D]$ where 
       $\psi^{\mathrm{an}}[m_D]$ is defined as for definable homotopies
       i.e. $(t,x) \mapsto \psi^{\mathrm{an}}(\mathrm{max}(t,m_D(x)),x)$.
        
       Let $h : [0,\infty] \times C \to \widehat{C}$ be a homotopy that is 
       the unique lift of a homotopy $\psi_D : [0,\infty] \times \mathbb{P}^1 \to \widehat{\mathbb{P}^1}$.  
       The results in \cite[\S 14.1]{HL} show that the canonical extension $H$ of $h$ which is a homotopy
       on $\widehat{C}$ descends to a homotopy 
       $C^{\mathrm{an}}$ via the map $\pi_{k,C}$. Since 
       the homotopies of $\widehat{C}$ were all unique lifts of the homotopies on $\widehat{\mathbb{P}^1}$,
       the induced homotopies on $C^{\mathrm{an}}$ will be the unique lifts of homotopies on 
       $\mathbb{P}^{1,\mathrm{an}}$ i.e. we have the following commutative diagram. 
       
       $$
 \begin{tikzcd} [row sep = large, column sep = large] 
I \times C^{\mathrm{an}} \arrow[r,"h^{\mathrm{an}}"] \arrow[d, "id \times f^{\mathrm{an}}"] &
 {C}^{\mathrm{an}} \arrow[d, "{f}^{\mathrm{an}}"] \\
I  \times \mathbb{P}^{1,\mathrm{an}} \arrow[r,"{\psi^{\mathrm{an}}_D}"] &
 \mathbb{P}^{1,\mathrm{an}} 
\end{tikzcd}
$$   
     where $I = [0,\infty]$. 
    
We are grateful to Jérôme Poineau for the proof of the following lemma.
      
\begin{lem} \label{fibres lemma}
    Recall that we 
have a map $\pi_{k,C} : \widehat{C}(k^{max}) \to C^{\mathrm{an}}$. 
Let $x \in C^{\mathrm{an}}$ be a point of type I or II. Then 
$\pi_{k,C}^{-1}(x)$ consists of precisely one point. 
If $x \in C^{\mathrm{an}}$ is of type III then $\pi_{k,C}^{-1}(x)$ is homeomorphic to an annulus
in $\mathbb{A}^{1,\mathrm{an}}_{k^{max}}$
whose skeleton consists of a single point. 
If $x$ is of type IV then $\pi_{k,C}^{-1}(x)$ is homeomorphic to a Berkovich closed disk 
contained in $\mathbb{A}^{1,\mathrm{an}}_{k^{max}}$. 
\end{lem}      
\begin{proof} 
The map $\pi_{k,C}$ factors through 
 $\widehat{C}(k^{max}) \to C_{k^{max}}^{\mathrm{an}}$ and the projection
  $\pi_{k,C}^{\mathrm{an}} : C_{k^{max}}^{\mathrm{an}} \to C^{\mathrm{an}}$.   
By \cite[Lemma 14.1.1]{HL}, $\widehat{C}(k^{max}) \to C_{k^{max}}^{\mathrm{an}}$
 is a homeomorphism.
 Hence we restrict our attention to proving the analogous statements for 
 the morphism $\pi_{k,C}^{\mathrm{an}}$.
 When $x$ is of type I, it is clear that $(\pi_{k,C}^{\mathrm{an}})^{-1}(x)$ 
 is a single type I point as well. 
 When $x$ is of type II, we have by 
 \cite[Theorem 2.15 and Remark 2.29]{PPII} that 
 $(\pi_{k,C}^{\mathrm{an}})^{-1}(x)$ decomposes into a type II point 
 $x_{k^{max}}$ and the disjoint union of Berkovich open balls. 
  Let $O$ be one such open ball. Since the $k^{max}$-points 
  are 
  dense in $C_{k^{max}}^{\mathrm{an}}$, there exists $z \in O(k^{max})$. 
  We must have that $\widetilde{\mathcal{H}(x)} \hookrightarrow \widetilde{\mathcal{H}(z)}$. 
  However this is not possible since $\mathcal{H}(z) = k^{max}$ and $\widetilde{k^{max}} = \widetilde{k}$. 
  
  Let $x \in C^{\mathrm{an}}$ be a point of type III. By \cite[Theorem 4.3.5]{D-up}, there exists
   a neighbourhood $X \subset C^{\mathrm{an}}$ of $x$ such that $X$ is isomorphic to a 
   Berkovich open annulus. We identify $X$ with an annulus in 
   $\mathbb{A}^{1,\mathrm{an}}_k$. After choosing a suitable coordinate $T$ on $\mathbb{A}^1_k$, we can assume that
   $X = \{z \in \mathbb{A}_k^{1,\mathrm{an}}| r < |T(z)| < R\}$ and that
   $x$ is the unique point in $X$ such that 
   $|T(x)| = r'$ where $r' \in (r,R) \smallsetminus |k^*|$.
   We check that $(\pi_{k,C}^{\mathrm{an}})^{-1}(X) \simeq X_{k^{max}}$ and 
   see that $(\pi_{k,C}^{\mathrm{an}})^{-1}(x) = \{z \in X_{k^{max}} | |T(z)| = r'\}$.
   
    If $x$ is of type IV, we use a similar argument noting in this case that
    there exists a neighbourhood $X$ of $x$ that is isomorphic to a Berkovich closed disk
    in $\mathbb{A}^{1,\mathrm{an}}_k$.
\end{proof}    
  
 \begin{lem} \label{fibres lemma skeleta}
 Let $C$ be a smooth projective curve over $k$ and 
 $f : C \to \mathbb{P}^1_k$ be a finite morphism.  
 Let $D \subset \mathbb{P}_k^1(k)$ be a divisor and 
 let $\psi_D : [0,\infty] \times \mathbb{P}^1 \to \widehat{\mathbb{P}^1}$
 be a $k$-definable homotopy whose image $\Upsilon_D$ is $\Gamma$-internal.
 Let $h : [0,\infty] \times C \to \widehat{C}$ be a $k$-definable homotopy which 
 is a lift of $\psi_D$ and whose image is $\Upsilon' \subset \widehat{C}$. 
 Let $\Sigma' = \pi_{k,C}(\Upsilon')(k^{max})$. 
 We then have that ${\pi_{k,C}}_{|\Upsilon'(k^{max})} : \Upsilon'(k^{max}) \to \Sigma'$ is a homeomorphism.
 \end{lem} 
  \begin{proof}
  Let $\Sigma := \pi_{k,\mathbb{P}^1}(\Upsilon_D)$. 
     We check using the explicit description of $\Upsilon_D$ that 
     the restriction 
     ${\pi_{k,\mathbb{P}^1}}_{|\Upsilon_D} : \Upsilon_D(k^{max}) \to \Sigma$ is bijective.
     By \cite[Lemma 14.1.2]{HL}, it is closed as well and hence a homeomorphism. 
  
By \cite[Lemma 14.1.2]{HL}
it suffices to show that 
  ${\pi_{k,C}}_{|\Upsilon'(k^{max})}$ is a bijection.
  Since $\widehat{f} : \Upsilon' \to \Upsilon_D$ is a finite map 
  and ${\pi_{k,\mathbb{P}^1}}_{|\Upsilon_D(k^{max})}$ is a bijection, we see 
  that 
  ${\pi_{k,C}}_{|\Upsilon'(k^{max})}$ is finite.
  
  We argue as in Lemma \ref{fibres lemma}.
  By loc.cit., we know that the map ${\pi_{k,C}}_{|\Upsilon'(k^{max})}$ 
  is bijective over points of type I and II. 
  Let $x \in \Sigma'$ be a point of type III 
  and let $X$ be a neighbourhood in $C^{\mathrm{an}}$ 
  which is homeomorphic to an open annulus in $\mathbb{A}^{1,\mathrm{an}}_k$.
  We identify $\Upsilon'(k^{max})$ with a subspace of 
  $C^{\mathrm{an}}_{k^{max}}$ via the homeomorphism 
  $\widehat{C}(k^{max}) \to C^{\mathrm{an}}_{k^{max}}$.
  Recall the map 
  $\pi_{k,C}^{\mathrm{an}} : C^{\mathrm{an}}_{k^{max}} \to C^{\mathrm{an}}$ and 
  observe that 
  $(\pi_{k,C}^{\mathrm{an}})^{-1}(X) \simeq X_{k^{max}}^{\mathrm{an}}$.  
  The map $X_{k^{max}} \cap \Upsilon'(k^{max}) \to X \cap \Sigma$ is finite.
  Suppose $z \mapsto x$. By Lemma \ref{fibres lemma}, the preimage of $x$ is an
  annulus whose skeleton is a single point $x'$. 
  Since $\Upsilon'(k^{max})$ is path connected, if $z \neq x'$ and $y$ is any other point in $\Upsilon'(k^{max})$ that doesn't map to $x$ then
   we must have a path 
  from $z$ to
  $y$. However, this path must necessarily pass through to $x'$ and contain an infinite number of preimages of $x$. 
  This is a contradiction and hence $x'$ is the unique point in the fibre over 
  $x$ that is contained in $\Upsilon'(k^{max})$.
  
  \end{proof}

 \subsection{Length of a definable path} \label{length of a definable path}
 
    Let $f : C' \to C$ be a finite separable morphism 
    of smooth projective connected curves. 
   Let $g : C \to \mathbb{P}^1_k$ be a finite separable morphism.  
  We deduce from \cite[Lemma 3.3]{WE2} that 
  there exists a divisor $D \subset \mathbb{P}^1$ such that
  if $\Sigma_{D}$ denotes the convex hull of $D$ in $\mathbb{P}^{1,\mathrm{an}}$ 
  then $\Sigma'_D := (g^{\mathrm{an}})^{-1}(\Sigma_D)$ and
  $\Sigma''_D := ((g\circ f)^{\mathrm{an}})^{-1}(\Sigma_D)$ are skeleta of $C^{\mathrm{an}}$ and $C'^{\mathrm{an}}$ respectively.
  We can enlarge $D$ if necessary to assume further that 
  the homotopy $\psi_D$ lifts uniquely via $f$ and $g \circ f$ to homotopies
  $h : [0,\infty] \times C \to \widehat{C}$
  and   $h' : [0,\infty] \times C' \to \widehat{C'}$
   with
   images $\Upsilon'_D = \widehat{f}^{-1}(\Upsilon_D)$
   and  $\Upsilon''_D = (\widehat{g \circ f})^{-1}(\Upsilon_D)$
    \footnote{One can 
   go through the Hrushovski-Loeser construction to show that there is actually no need to enlarge $D$.}.
  
   Let $x \in C'(k)$ 
   and $y := f(x)$. 
   Let $h'^{\mathrm{an}}_x : [0,\infty] \to C'^{\mathrm{an}}$ denote the
   path defined by 
   $t \mapsto h'^{\mathrm{an}}(t,x)$.
   Recall the isometric paths 
 $l_x : (0,\infty] \to C'^{\mathrm{an}}$ 
 and $l_y : (0,\infty] \to C^{\mathrm{an}}$
 from \S 1. 
 The map $f^{\mathrm{an}}$ induces a map 
 $f_x^{\mathrm{an}} = l_y^{-1} \circ f^{\mathrm{an}} \circ l_x : (0,\infty] \to (0,\infty]$.
 There exists a tuple $T_x := ((t_0 = \infty,d_0,\alpha_0),\ldots,(t_{m_x},d_{m_x},\alpha_{m_x})) \subset ([0,\infty] \times \mathbb{N} \times [0,\infty])^{m_x}$ 
  such that for every $i$, $f_x^{\mathrm{an}}$ restricted to 
 $(t_i,t_{i+1}]$ is the map $t \mapsto d_it + \alpha_i$.
 
  Suppose $x \notin \Sigma''_D$. 
  Let $O'$ be the connected component of $C'^{\mathrm{an}} \smallsetminus \Sigma''_D$ containing $x$ and 
  $O$ be the connected component of $C^{\mathrm{an}} \smallsetminus \Sigma'_D$ containing $y = f(x)$.
  We identify $O'$ and $O$ with the Berkovich open unit ball so that $x$ is the origin on $O'$ and 
  $y$ is the origin on $O$. 
  The function $f^{\mathrm{an}}$ restricts to a map $O' \to O$. 
  Let $F : O' \to \mathbb{R}_\infty$ be the function 
  given by
  $p \mapsto -\mathrm{log}|T(f(p))|$ where $T$ is the coordinate on $O$. 
  By Lemma \ref{piecewise affine and zeroes}, $F$ is piecewise affine 
  as defined in \cite[Definition 5.14]{BPR}.
  Using the notation from the previous paragraph, 
  observe that for every $i$, the integer $d_i$ 
  is the \emph{outgoing slope} (cf. \cite[Definition 5.14]{BPR}) of the piecewise affine function 
   $F$
    along the tangent direction $v$ towards the boundary of the 
    ball containing $x$. We denote this $d_vF$.

      In \S \ref{definable subsets of C}, we introduced the notion 
    of the backward branching index $BB_{h',f}$ associated to the homotopy 
    $h' : [0,\infty] \times C' \to \widehat{C'}$ and the morphism $f : C' \to C$. 
    Analogously, we define $BB_{h'^{\mathrm{an}},f} : C'(k) \to \mathrm{Fn}([0,\infty] \to \mathbb{N} \times [0,\infty])$ where 
    $\mathrm{Fn}([0,\infty] \to \mathbb{N} \times [0,\infty])$ denotes the set of definable functions from 
    $[0,\infty]$ to $\mathbb{N} \times [0,\infty]$. 
   Let $z \in C'(k)$ and $t \in [0,\infty]$. Let $B_{h'^{\mathrm{an}},g}(x)(t)$ denote the set of 
          elements $z \in C'$ such that 
        $g(z) = g(x)$ and 
        $h'^{\mathrm{an}}(t,z) = h'^{\mathrm{an}}(t,x)$.
        Then 
        \begin{align*}
        BB_{h'^{\mathrm{an}},g}(x)(t) := (\mathrm{card}(B_{h'^{\mathrm{an}},g}(x)(t)), \Sigma_{z \in B_{h'^{\mathrm{an}},g}(x)(t)} \mathrm{val}(z)).
        \end{align*}

    One checks using that if $z \in C'(k)$ then 
    $BB_{h',f}(z) = BB_{h'^{\mathrm{an}},f}(z)$.
    This follows from the following fact which can be
    deduced from Lemma \ref{fibres lemma}.  
    For every $z_1,z_2 \in C'(k)$ and $t \in [0,\infty](\mathbb{R})$, 
    $h'(t,z_1) = h'(t,z_2)$ if and only if 
    $h'^{\mathrm{an}}(t,z_1) = h'^{\mathrm{an}}(t,z_2)$. 
    Hence 
     $BB_{h'^{\mathrm{an}},f}(z) \in \mathrm{Fn}([0,\infty] \to \mathbb{N} \times [0,\infty])$.
       
     We now show that the backward branching index
     associated to the homotopy $h^{\mathrm{an}}$ 
     coincides with the notion of outgoing slopes.  
       
  \begin{lem} \label{slopes and backward branching index}
   We make use of the notation introduced above regarding the morphism $f : C' \to C$ and 
   the homotopy $h' : [0,\infty] \times C' \to \widehat{C'}$. 
   Let $x \in C'(k) \smallsetminus \Sigma''_D$ and $p = l_x(t)$ where $t \in [0,\infty]$. 
   Let $v$ be the tangent direction along $l_x$ towards $l_x(0)$. 
   We have that 
   $$\mathrm{card}(B_{h'^{\mathrm{an}},f}(x)(t)) = d_vF(p).$$
  \end{lem}  
  \begin{proof} 
   Let $y := f(x)$. Let $O'$ be the connected component of 
   $C'(k) \smallsetminus \Sigma''_D$ containing $x$
    and $O$ be the connected component of 
   $C(k) \smallsetminus \Sigma'_D$ containing $y = f(x)$. 
   We identify $O'$ and $O$ with the Berkovich open unit ball $O(0,1)$ 
   so $x$ and $y$ are the origins of the source and target respectively.
   We may hence identify the restriction of $f^{\mathrm{an}}$ with 
   an endomorphism of the Berkovich unit disk. 
   The function $F$ was defined to be
   $p \mapsto -\mathrm{log}|T(f(p))|$ where $T$ is the coordinate on the target. 
   Let 
   $q := f^{\mathrm{an}}(p)$. 
   Let $b_0$ be the germ of the outgoing path from $q$ to the origin of $O$ and 
   $b_{\infty}$ be the germ of the outgoing path from $q$ to the boundary of $O$.  
   Observe that $T$ is constant along every outgoing path different from $b_0$ and 
   $b_{\infty}$.  
   It follows that if $\mathcal{T}(p)$ is the tangent space at $p$ i.e. the 
   space of all outgoing paths from $p$ then  
   $$\Sigma_{t \in \mathcal{T}(p)} d_tF(p) = \Sigma_{w \in \mathcal{V}_0} d_wF(p) + \Sigma_{u \in \mathcal{V}_\infty} d_uF(p)$$
   where $\mathcal{V}_0$ is the set of outgoing paths from $p$ that map onto 
   $b_0$ and $\mathcal{V}_\infty$ is the set of outgoing paths from $p$ that map onto 
   $b_\infty$. By uniqueness of the lift, $\mathcal{V}_{\infty}$ consists of the single outgoing path 
   from $p$ to $\infty$. Observe that this is the same path $v$ in the statement of the lemma. 
   Since $F$ is harmonic, we see that 
   $$d_vF(p)  = -\Sigma_{w \in \mathcal{V}_0} d_wF(p).$$
   Let $p$ be of the form $\zeta_{0,r}$ where $r \in (0,1)$. 
   By Lemma \ref{piecewise affine and zeroes}, we get that
   $d_vF(p)$ is the sum of 
   the order of vanishing of $f$ at every zero in 
   the closed ball of radius $r$ around the origin i.e.
   $$d_vF(p)  = \Sigma_{z \in Z_f(r)}  \mathrm{ord}_z(f)$$
   where $Z_f(r)$ denotes the zeroes of $f$ in the closed ball 
   around $0$ of radius $r$. 
   
   To complete the proof, we show that
    $$\mathrm{card}(B_{h'^{\mathrm{an}},f}(x)(t)) = \Sigma_{z \in Z_f(r)} \mathrm{ord}_{z}(f).$$
    
    The homotopies $h'$ and $h$ are the unique lifts of a homotopy 
    $\psi_D$. The construction in \cite[\S 7.5]{HL}, 
    implies that $D$ must contain the image of the ramification loci 
    of both $g$ and $g \circ f$. 
    We deduce from this that $f$ is unramified outside $(g \circ f)^{-1}(D)$. 
    It follows that 
    for every $z \in Z_f(r)$, $\mathrm{ord}_{z}(f) = 1$. 
    If $z$ does not belong to the closed ball around $0$ of radius $r$ then 
    $h'(t,z) \neq p$.
    It follows that $B_{h'^{\mathrm{an}},f}(x)(t) \subset Z_f(r)$ and hence 
        \begin{align} \label{key}
    \mathrm{card}(B_{h'^{\mathrm{an}},f}(x)(t)) \leq \Sigma_{z \in Z_f(r)} \mathrm{ord}_{z}(f)
    \end{align}
    
    We show that if $z \in Z_f(r)$ then $h'^{\mathrm{an}}(t,z) = p$. 
   Since $h'^{\mathrm{an}}$ is a lift of the homotopy 
   $h^{\mathrm{an}}$ and $h^{\mathrm{an}}(t,0) = q$, we get that 
   $f^{\mathrm{an}}h'^{\mathrm{an}}(t,z) = q$. 
  Furthermore, we know that $f_{z}^{\mathrm{an}}$ maps the image of $l_z$ onto $l_0$ and hence 
  the lift starting from $z$ of the path $h^{\mathrm{an}}(-,0) : [0,\infty] \to C^{\mathrm{an}}$ must be along 
 $l_z$. 
 By assumption, $z$ lies in the closed ball around $0$ of radius $r$ and hence we get 
 that $p \in l_z$. 
 Since $f^{\mathrm{an}}$ is injective when restricted to the 
  image of $l_z$, we
 see that 
  $h'^{\mathrm{an}}(t,z) = p$.
  This proves that 
  $$\Sigma_{z \in Z_f(r)} \mathrm{ord}_{z}(f) \leq \mathrm{card}(B_{h'^{\mathrm{an}},f}(x)(t))$$
  and by (\ref{key}), we conclude the proof. 
    \end{proof} 

     Recall from \S 1 that we used $O(0,1)$ to denote the Berkovich open unit disk.  
       
 \begin{lem} \label{piecewise affine and zeroes}
    Let $f : O(0,1) \to O(0,1)$ be a clopen surjective analytic function
    that maps $0$ to $0$.
  Let
  $F$ be the function
   $p \mapsto -\mathrm{log}|T(f(p))|$ where $T$ is the coordinate on the target. 
   The function $F$ is piecewise affine and 
   harmonic (cf. \cite[Definition 5.14]{BPR}).
  
   Let 
   $r \in (0,1)$ and $p := \zeta_{0,r}$.  
  Let $v$ be
   an outgoing branch from $p$ that maps onto 
   the outgoing branch from $f(p)$ to the origin.
   We then have that the outgoing slope $d_vF(p)$
   is equal to the sum $\Sigma_{z \in Z_f(r,v)} -\mathrm{ord}_z(f)$ where 
$Z_f(r,v)$ denotes those elements $z \in Z_f(r)$
   that belong to 
  the open ball of radius $r$ that intersects $v$. 
 \end{lem}    
       \begin{proof} 
    Let $B' \subset O(0,1)$ be a Berkovich closed disk around $0$ and 
    $B := f(B')$. 
    By \cite[Lemma 2.33]{WE2}, $B$ is a Berkovich closed disk
    and we identify $B'$ and $B$ with the Berkovich closed unit disk around the origin.
     The function $f : B' \to B$ corresponds to an element $f \in k\{S\}$.
     By Weierstrass preparation, 
     we get that 
     $f = c\prod_j (S - a_j)u$ where $u \in k\{S\}$ is invertible 
     and $|s(u)| = 1$ for all $s \in B'$. 
     The map 
     $F_{|B'}$ coincides with the map 
     $s \mapsto \mathrm{val}(c) + \Sigma_j -\mathrm{log}(|(S-a_j)(s)|)$ and this is well known to be piecewise affine. 
     Since $B'$ can be taken to be arbitrarily large in $O(0,1)$, we get that 
     $F$ is piecewise affine and harmonic on $O(0,1)$.  
  
       We may assume without loss of generality that 
       $v$ is the outgoing path from $p$ to the origin. 
       By \cite[Theorem 5.15]{BPR}, we get that 
       $$d_vF(p) = -\mathrm{ord}_v(\widetilde{f_p})$$ where 
       $\widetilde{f_p}$ is the reduction as defined in 
       \S 5.13 of loc.cit. We can calculate $\widetilde{f_p}$ using the explicit description of $f$ to get the result. 
             \end{proof} 
             
     As before,  let $x \in C'(k)$ and $y := f(x)$. 
  Recall that the map $f^{\mathrm{an}}$ induces a map 
 $f_x^{\mathrm{an}} = l_y^{-1} \circ f^{\mathrm{an}} \circ l_x : (0,\infty] \to (0,\infty]$.
 There exists a tuple $$T_x := ((t_0 = \infty,d_0,\alpha_0),\ldots,(t_{m_x},d_{m_x},\alpha_{m_x})) \subset ([0,\infty] \times \mathbb{N} \times [0,\infty])^{m_x}$$ 
  such that for every $i$, $f_x^{\mathrm{an}}$ restricted to 
 $(t_i,t_{i+1}]$ is the map $t \mapsto d_it + \alpha_i$.
 
    Recall from section \S \ref{backward branching index}, 
    that the function 
    $BB_{h'^{\mathrm{an}},f}(x) : [0,\infty] \to \mathbb{N} \times [0,\infty]$ is 
    determined completely by the tuple 
    $S_x := ((s_0 = \infty,e_0,\beta_0),\ldots,(s_{n_x},e_{n_x},\beta_{n_x}))$ where 
     the 
    $s_i$ are the break points i.e. for every 
    $i$, $e_i \neq e_{i+1}$.

 \begin{lem} \label{weierstrass prep}
  Let $f : O(0,1) \to O(0,1)$ be a finite surjective analytic morphism such that 
  $f(0) = 0$. There exists a real number $r_0 \in (0,1)$ such 
  that for every $r \geq r_0$, the restriction of $f$ to the ball $O(0,r)$ is of the form 
  $T \mapsto |T-\beta_1|\ldots|T-\beta_k|$ where 
  $\beta_1,\ldots,\beta_k$ are the preimages of $0$. 
 \end{lem}    
    \begin{proof} 
    Let $r \in (0,1)$ and    
    $a \in k^*$ such that $|a| = r$. 
    The function $T \mapsto f(aT)$ is an analytic function 
    on the closed unit disk $B(0,1)$. It follows
    from the Weierstrass preparation theorem
     that 
    $$f(aT) = c(r)(T-\alpha_1)\ldots(T-\alpha_s)u(T)$$ 
    where $u$ is a unit on $B(0,1)$ and $c(r) \in k^*$.
    By applying the transformation $T =  T/a$, we see 
    that  if $0 \leq |T| \leq r$ then 
    $$f(T) = c(r)/a^s(T-a\alpha_1)\ldots(T-a\alpha_s)u(T/a).$$
    Thus for $x \in O(0,r)$, 
    $$|f(T)(x)| = |c(r)/a^s||(T-a\alpha_1)(x)|\ldots|(T - a\alpha_s)(x)||u(T/a)(x)|.$$
    Since $u(T)$ is a unit on $B(0,1)$ and $T \mapsto T/a$ is an isomorphism 
    from $O(0,r)$ to $O(0,1)$ we get that $u(T/a)$ is a unit on $O(0,r)$ and hence 
    $$|f(T)(x)| = |c(r)/a^s||(T-a\alpha_1)(x)|\ldots|(T - a\alpha_s)(x)|.$$
    
       By \cite[Lemma 2.2.5]{TEM3}, $f_0 := l_0^{-1}fl_0$ is piecewise affine with integer slopes
      where $l_0$ is the isometric embedding of $(0,\infty]$ into $O(0,1)$ as explained in \S \ref{introduction}.
       It should be noted that the notation employed in Temkin is multiplicative while we 
       write the value group additively. 
        It follows that there exists $t_0 \in [0,\infty]$ such that 
        for $t \in (0,t_0]$, $f_0(t) = dt + \alpha$ for some $\alpha \in \mathrm{val}(k^*)$.
        Since $f_0$ maps $O(0,1)$ bijectively onto itself, we must have that 
        $\alpha = 0$. 
       Hence for $r > \mathrm{exp}(-t_0)$ and $r > |a\alpha_i|$ for every $i$, we must have that
       $|c(r)/a^s| = 1$. Setting $\beta_i = \alpha_i$ completes the proof of the lemma. 
      \end{proof} 
      
 \begin{cor} \label{corollary weirstrass prep}
   Let $f : O(0,1) \to O(0,1)$ be a finite surjective étale analytic morphism such that 
  $f(0) = 0$. As before, let $f_0 := l_0^{-1} \circ f \circ l_0$ where 
  $l_0$ is the isometric embedding $(0,\infty] \to O(0,1)$. 
  Let 
  $$T_0 := ((t_0 = \infty,d_0,\alpha_0),\ldots,(t_{m_x},d_{m},\alpha_{m})) \subset ([0,\infty] \times \mathbb{N} \times [0,\infty])^{m}$$
   be  
  such that for every $i$, $f_0$ restricted to 
 $(t_i,t_{i+1}]$ is the map $t \mapsto d_it + \alpha_i$. 
  Let $\mathcal{B} := \{\beta_0  = 0,\ldots,\beta_n\}$ be the preimages of $0$
  such that $|\beta_0| < |\beta_1| \leq \ldots \leq |\beta_n|$ and 
  $|\beta_{i_0} = 0| < |\beta_{i_1}| \ldots < |\beta_{i_l}|$ be the 
  distinct values of the set 
  $\{|\beta_0|,\ldots,|\beta_n|\}$. We then have that 
  $m = i_l$ and for every $0 \leq j \leq m$, the following hold. 
  \begin{enumerate}
  \item $t_j = -\mathrm{log}|\beta_{i_j}|$.
  \item $d_{j} = \mathrm{card}\{\beta \in \mathcal{B} | |\beta| \leq |\beta_{i_j}|\}$. 
  \item $\alpha_j =  \Sigma_{\beta \in \mathcal{B}}( -\mathrm{log}|\beta|)  - \Sigma_{\beta \in \mathcal{B}, |\beta| \leq |\beta_{i_j}|} (-\mathrm{log}|\beta|)$   
      \end{enumerate}
   \end{cor}  
 \begin{proof} 
   The corollary is deduced by explicit computation using the description provided by 
   Lemma \ref{weierstrass prep}. 
  \end{proof}
    
        \begin{rem} 
   \emph{It might happen that for any $x \in C'$, there exists $t_0 \in [0,\infty]$ such that 
   $h'(-,x)$ is constant on $[0,t_0]$. As a consequence, $s_{i} \geq t_0$. The fact that 
   the $s_i$ take values in $(t_0,\infty]$ while the $t_i$ take values in $(0,\infty]$ 
   need to be taken into account when relating them.  
     Hence we introduce an invariant that will help relate the Backward branching index 
   with the profile functions. Given $x \in C'$, let $s_{n_x + 1} \in [0,\infty]$ 
   be the largest element such that $h'(-,x)$ is constant 
   when restricted to the interval 
   $[0,s_{n_x + 1}]$. By definition of $BB_{h',x}$, we have that 
   $s_j \geq s_{n_x + 1}$ for every $1 \leq j \leq n_x$.}
    \end{rem}  
 
 \begin{lem} \label{T_x and S_x are intercalculable}
We make use of the notation introduced above. Suppose that
$C = \mathbb{P}^1$ and $g$ is the identity map. 
Let $x \in C'(k)$.
We then have that $n_x = m_x$
and
for every $0 \leq j \leq m_x = n_x$, the following equalities hold.
\begin{enumerate}   
\item $s_j - s_{n_x + 1} = d_{j}t_j + \alpha_{j}$
\item $e_j = d_j$
\item $\alpha_j = \beta_{m_x} - \beta_j$
\end{enumerate}
   Lastly, if $h'^{\mathrm{an}}_x : [0,\infty] \to C'^{\mathrm{an}}$ denotes the
   path defined by 
   $s \mapsto h'^{\mathrm{an}}(s,x)$
then
   for every 
   $a,b \in [s_i,s_{i+1})$ with $a < b$ and $1 \leq i \leq n_x$, 
   we have that $\rho(h'^{\mathrm{an}}_x(a),h'^{\mathrm{an}}_x(b)) = (1/d_i)(b-a)$
    where $\rho$ is the path distance metric on $C'^{\mathrm{an}}$. 
 \end{lem}       
\begin{proof} 
Let $T_x = ((t_0 = \infty,d_0,\alpha_0),\ldots,(t_{m_x},d_{m_x},\alpha_{m_x}))$. 
Recall that the function $F$ given by $t \mapsto -\mathrm{log}|T(f(p))|$ is piecewise affine 
and its breakpoints coincide with set 
$\{t_0,\ldots,t_{m_x}\}$. 
Lemma \ref{slopes and backward branching index} relates the
 break points of the backward branching index $BB_{h'^{\mathrm{an}},f}(x)$
with the break points of $F$ and we deduce that 
\begin{align}
T_x = ((l_x^{-1}{h'}_x^{\mathrm{an}}(s_0),e_0,\alpha_0),\ldots,(l_x^{-1}{h'}_x^{\mathrm{an}}(s_{n_x}),e_{n_x},\alpha_{n_x}))
\end{align}
In particular, we see
that $n_x = m_x$ and verify
 assertion (2) i.e. $e_i = d_i$.
    Let $t = l_x^{-1}h'^{\mathrm{an}}_x(s)$ for $s \in (s_{n_x + 1},\infty]$.
        We make the following sequence of deductions. 
    \begin{align} 
     l_x(t) &= h'^{\mathrm{an}}_x(s). \\
     f^{\mathrm{an}} \circ l_x(t) &= f^{\mathrm{an}} \circ h'^{\mathrm{an}}_x(s).
    \end{align} 
 Let $y := f(x)$.
By 
construction, the homotopy $h'$ is the lift of a homotopy
$\psi_D : [0,\infty] \times \mathbb{P}^1 \to \widehat{\mathbb{P}^1}$. 
 The path $\psi^{\mathrm{an}}_{D,y} : (0,\infty] \to \mathbb{P}^{1,\mathrm{an}}$ 
 differs from the path $l_y$ by a scaling factor. 
  One checks that 
  $$\psi^{\mathrm{an}}_{D,y} =  l_y \circ (s \mapsto s - s_{n_x + 1}).$$
Since $h'$ is the lift of $\psi_D$, 
\begin{align*} 
f^{\mathrm{an}} \circ h'^{\mathrm{an}}_x(s) &= \psi_{D,y}^{\mathrm{an}}(s) \\ 
              &= l_y(s - s_{n_x + 1}).
              \end{align*}

One deduces from (4) above that 
for every $s \in (s_{n_x + 1},\infty]$, 
\begin{align}
s - s_{n_x + 1} = f_x^{\mathrm{an}}(t).
\end{align} 
It follows from (2) that $s_i - s_{n_x + 1} = f_x^{\mathrm{an}}(t_i)$. 
By definition of the invariants 
$d_i$ and $\alpha_{i}$, $f_x^{\mathrm{an}}(t_i) = d_{i}t_i + \alpha_{i}$. 
We may thus conclude a proof for assertion (1) of the lemma.
Lastly, assertion (3) from the Lemma follows from Corollary \ref{weierstrass prep}.

We now prove the last part of the lemma. 
By Equation (2),
\begin{align}
 \rho({h'}_x^{\mathrm{an}}(a),{h'}_x^{\mathrm{an}}(b)) &= \rho(l_x(f_x^{\mathrm{an}})^{-1}(a - s_{n_x + 1}),l_x(f_x^{\mathrm{an}})^{-1}(b - s_{n_x + 1})) \\
               &= |(f_x^{\mathrm{an}})^{-1}(a - s_{n_x + 1}) - (f_x^{\mathrm{an}})^{-1}(b - s_{n_x + 1})|.
\end{align} 

By Equation (5) and since $a,b \in [s_i,s_{i+1})$, we must have that 
$$(f_x^{\mathrm{an}})^{-1}(a - s_{n_x + 1}), (f_x^{\mathrm{an}})^{-1}(b - s_{n_x + 1}) \in [t_i,t_{i  + 1}).$$
Recall that if $t \in [t_i,t_{i+1})$ then 
$f_x^{\mathrm{an}}(t) = d_it + \alpha_i$. 
It follows that 
$$(f_x^{\mathrm{an}})^{-1}(a - s_{n_x + 1}) = (1/d_i)(a - s_{n_x + 1}) - (1/d_i)(\alpha_i)$$
and 
$$(f_x^{\mathrm{an}})^{-1}(b - s_{n_x + 1}) = (1/d_i)(b - s_{n_x + 1}) - (1/d_i)(\alpha_i).$$

A simple calculation using Equation (7) now completes the proof. 


\end{proof}        
       
 \subsection{Radiality}       
       
Lemma \ref{T_x and S_x are intercalculable} tells us that the tuples 
$T_x$ and $S_x$ are equivalent as invariants of the morphism $f$ and the deformation $h'$. 
We use this fact along with Theorem \ref{backward branching is constant along retractions} to 
show the following result which is also proved in \cite[Theorem 3.3.11]{TEM3}.

\begin{thm} \label{tameness of profile functions}
  Let $f : C' \to C$ be a finite separable morphism of smooth projective irreducible $k$-curves.  
  There exists deformation retractions 
  $h'^{\mathrm{an}} : [0,\infty] \times C'^{\mathrm{an}}\to C'^{\mathrm{an}}$ and 
  $h^{\mathrm{an}} : [0,\infty] \times C^{\mathrm{an}} \to C^{\mathrm{an}}$ 
  which are compatible with the morphism $f^{\mathrm{an}}$ and satisfy the following properties. 
  \begin{enumerate} 
  \item The images $\Sigma' := h'(0,C'^{\mathrm{an}})$ and $\Sigma = h(0,C^{\mathrm{an}})$ are
   skeleta of the curves $C'^{\mathrm{an}}$ and $C^{\mathrm{an}}$. 
  \item We make use of the notation from \S \ref{length of a definable path}.
  The functions $f_x^{\mathrm{an}}$ are constant along the fibres of the retraction. 
  Equivalently,
  for every $x,y \in C'(k)$,
if  $h'^{\mathrm{an}}(x) = h'^{\mathrm{an}}(y)$ then 
  $T_x = T_y$. We write $T^f_x$ in place of $T_x$ to emphaisize the role of the morphism $f$.
   \item There exists a function $T^f : \Sigma' \to \bigcup_{M \in \mathbb{N}} ((0,\infty] \times \mathbb{N})^{M}$ given by 
   $p \mapsto T^f_p$ where if
   $T^f_p  = ((t_0 = \infty,d_0,\alpha_0),\ldots,(t_{m_p},d_{m_p},\alpha_{m_p}))$ 
  and $x$ retracts to $p$ via $h'^{\mathrm{an}}$ then $T^f_x = T^f_p$.
  Furthermore, the function $T^f$ is definable by which we mean the following. 
If $\Sigma'_m$ denotes the locus of points $p \in \Sigma'$
  where $T^f_p$ is an $m$-tuple then 
  $\Sigma'_m$ is the union of finitely many segments whose end points are points of 
  type I or II 
   and  
  for every $0 \leq i \leq m$, the function $p \to d_i$ is constant on each such segment while
  $p \mapsto \alpha_i$ is 
  piecewise affine on 
  $\Sigma'_m$.  
\end{enumerate} 
\end{thm} 
\begin{proof} 
 Let $g : C \to \mathbb{P}^1$ be a finite separable morphism. 
 We choose a finite subset $D \subset \mathbb{P}^1(k)$ such that if $\Sigma_D \subset \mathbb{P}^{1,\mathrm{an}}$ is 
 the convex hull of $D$ then the following are satisfied. 
 \begin{enumerate}
 \item Let $f' := g \circ f$. The preimages $\Sigma' := (f'^{\mathrm{an}})^{-1}(\Sigma_D)$ and 
 $\Sigma := (f^{\mathrm{an}})^{-1}(\Sigma_D)$ are skeleta of the curves 
 $C'^{\mathrm{an}}$ and $C^{\mathrm{an}}$ respectively. 
 \item The homotopy $\psi_D : [0,\infty] \times \mathbb{P}^1 \to \widehat{\mathbb{P}^1}$ lifts uniquely to homotopies
 $h' : [0,\infty] \times C' \to \widehat{C'}$ and $h : [0,\infty] \times C \to \widehat{C}$ via the morphisms 
 $f'$ and $f$ respectively. 
 \item The functions $BB_{h',f'} : C' \to \mathrm{Fn}([0,\infty],\mathbb{N})$ and 
 $BB_{h,g} : C \to  \mathrm{Fn}([0,\infty],\mathbb{N})$ are constant along the fibres of the retraction. 
 \end{enumerate}
    Assertion (1) follows from the \cite[Lemma 3.3]{WE2}. Assertion (2) is justified in 
    \S \ref{homotopies of C}. Assertion (3) is a consequence of 
    Theorem \ref{backward branching is constant along retractions} and 
    Remark \ref{treating multiple morphisms}. 
   By Lemma \ref{T_x and S_x are intercalculable}, we deduce from assertion (3)  
    that if $x',y' \in C'^{\mathrm{an}}$ are such that $h'^{\mathrm{an}}(0,x') = h'^{\mathrm{an}}(0,y')$ then 
    $T^{f'}_{x'} = T^{f'}_{y'}$.
    Likewise, if $x,y \in C^{\mathrm{an}}$ are such that $h^{\mathrm{an}}(0,x) = h^{\mathrm{an}}(0,y)$ then 
    $T^{g}_{x} = T^{g}_{y}$.
  
   Let $x' \in C'$ and $x := f(x')$. 
   Let 
   \begin{align*} 
   T^{f'}_{x'} &:= \{(t_0,d_0,\alpha_0),\ldots,(t_{n_1},d_{n_1},\alpha_{n_1})\} \\  
    T^g_{x} &:= \{(s_0,e_0,\beta_0),\ldots,(s_{n_2},e_{n_2},\beta_{n_2})\} \mbox{ and} \\ 
    T^f_{x'} &:= \{(u_0,c_0,\gamma_0),\ldots,(u_{n_3},c_{n_3},\gamma_{n_3})\}.
    \end{align*} 
    These tuples are such that if
   $t \in (t_i,t_{i+1}]$ then we have that 
   ${f'}_{x'}^{\mathrm{an}}(t) = d_it_i + \alpha_i$. 
   Likewise, if $s \in (s_i,s_{i+1}]$ we have that 
   $g_{x}^{\mathrm{an}}(s) = e_is_i + \beta_i$ and if $u \in (u_i,u_{i+1}]$ we have that 
   $f_{x'}^{\mathrm{an}}(u) = c_iu_i + \gamma_i$. 

%
  
    Observe that if 
     $x'_1,x'_2 \in C'$ are such that $h'^{\mathrm{an}}(0,x'_1) = h'^{\mathrm{an}}(0,x'_2)$ then 
     $T^f_{x'_1} = T^f_{x'_2}$. 
     By construction, we have that 
     $T^{f'}_{x'_1} = T^{f'}_{x'_2}$ and 
     if $T^{g}_{f(x'_1)} = T^g_{f(x'_2)}$. 
      By \cite[Lemma 2.3.8]{TEM3}, $T^f_{x'}$ is determined completely by 
   $T^g_x$ and $T^{f'}_{x'}$. 
   
       To complete the proof, we must show part (3) of the theorem. 
      By Lemma \ref{T_x and S_x are intercalculable} and 
      since $S_x$ is definable in $x$, we get that part (3) 
      is true for the functions 
     $T^g : \Sigma' \to \bigcup_{M \in \mathbb{N}} ((0,\infty] \times \mathbb{N})^{n_1}$ 
     and $T^{g \circ f} : \Sigma' \to \bigcup_{M \in \mathbb{N}} ((0,\infty] \times \mathbb{N})^{n_2}$.  
\end{proof} 

\begin{rem} 
  \emph{We preserve our notation from Theorem \ref{tameness of profile functions}.  The proof above shows that 
  $T^f_x$ can be calculated from $S^f_x$ and hence so can the functions $p \mapsto m_p$ 
  and for every $i$, $p \mapsto d_i$ and $p \mapsto \alpha_i$.}
\end{rem} 

   We now prove Theorem \ref{definable subsets are radial} which was discussed in \S \ref{introduction}. 
   Recall that our goal is to show that if given a closed definable subset $X$ of the analytification of a smooth projective 
   curve $C^{\mathrm{an}}$, there exists a skeleton $\Sigma$ of $C^{\mathrm{an}}$ 
   such that $X$ is piecewise affine radial around $\Sigma$ (cf. Definition \ref{radial set}). 
    The key ingredient in the proof is Theorem \ref{radiality of definable sets} which 
    proves a version of the radiality statement for Hrushovski-Loeser curves.  
    
\begin{proof} (Theorem \ref{definable subsets are radial}). 
  Let $X \subset C^{\mathrm{an}}$ be a closed path connected
   $k$-definable set (cf. Definition \ref{definable subset of an analytic curve}). By definition, there exists a closed 
 $k$-definable set $X' \subset \widehat{C}$ such that $X'$ is closed and $\pi_{k,C}(X') = X$. 
 Let $f : X \to \mathbb{P}^1$ be a finite morphism. By 
 Theorem \ref{tameness of profile functions}, there exists a divisor $D \subset \mathbb{P}^1$ such that 
 the deformation $\psi_D$ lifts uniquely to a deformation retraction 
 $h : [0,\infty] \times C \to \widehat{C}$, the image of the deformation 
 $h^{\mathrm{an}}$ is a skeleton of $C^{\mathrm{an}}$ and
  for every 
 $x \in C(k)$, the function $f_x^{\mathrm{an}}$ depends only on the point 
 $h^{\mathrm{an}}(0,x)$.
 Let $\beta_{X',H} : C \to [0,\infty] \times \{\gamma_1,\gamma_2\}$ be as in \S \ref{radiality}. 
 By Theorem \ref{radiality of definable sets}, we can enlarge $D$ and assume that 
 $\beta_{X',H}$ is constant along the fibres of the retraction $h(0,-)$. 
 By Theorem \ref{backward branching is constant along retractions}, 
 we can further enlarge $D$ and assume that 
 the function $BB_{h,f}$ is constant along the fibres of the retraction $h$.
 Let $\Upsilon \subset \widehat{C}$ denote the image of the deformation retraction 
 $h$. We have that 
 $\beta_{X',H}$ factors through a definable function 
 $\beta' : \Upsilon \to [0,\infty] \times \{\gamma_1,\gamma_2\}$. 
 Recall from the construction of $h^{\mathrm{an}}$ and Lemma \ref{fibres lemma} that
 for $x,y \in C(k)$, $h(0,x) = h(0,y)$ if and only if
 $h^{\mathrm{an}}(0,x) = h^{\mathrm{an}}(0,y)$. We deduce that 
  $[\beta_{X',H}]_{|C(k)}$ factors through the retraction 
 $h^{\mathrm{an}}(0,-)$. Let $\Sigma$ denote the image of the 
 homotopy $h^{\mathrm{an}}$ and let
 $\beta : \Upsilon(k^{max}) \simeq \Sigma \to \Gamma_\infty$ \footnote{The homeomorphism above is due to Lemma \ref{fibres lemma skeleta}.}
 be the 
 function defined as follows. We set $\beta(z) := p_1\beta'(z)$ if $\beta'(z) \neq (e(z),\gamma_2)$
 and $\beta(z) := \gamma_2 < 0$ if $\beta'(z) = (e(z),\gamma_2)$
 where the notation $e(z)$, $\gamma_1$ and $\gamma_2$ is as in \S \ref{radiality}.

 We make use of the notation introduced above Definition \ref{radial set}. 
 More precisely, we refer to the function 
 $r_{\Sigma} : C^{\mathrm{an}} \to [0,\infty]$ that measures the radius around the 
 skeleton $\Sigma$.
  Let $x \in C^{\mathrm{an}}$ and $\gamma := h^{\mathrm{an}}(0,x)$. We must show that there exists 
 $p_X(\gamma) \in [0,\infty]$ such that 
 $x \in X$ if and only if $r_\Sigma(x) \leq p_X(\gamma)$ where 
 $\gamma = h^{\mathrm{an}}(0,x)$.
 Let us first assume that $x$ is not of type IV. 
 We see that there exists a point $z \in C(k)$ such that 
 $x$ belongs to the image of the path $h_z^{\mathrm{an}}$ where 
 $h_z^{\mathrm{an}}$ was defined in 
 Lemma \ref{T_x and S_x are intercalculable}. 
By construction, 
$h_z^{\mathrm{an}}(t) \in X$ if and only if $t \in [0,\beta(\gamma)]$. 
Lemma \ref{T_x and S_x are intercalculable} allows us to calculate the 
distance 
$\rho(\gamma = h_z^{\mathrm{an}}(0),h_z^{\mathrm{an}}(\beta(\gamma)))$ using only 
the Backward branching index $BB_{h^{\mathrm{an}},f}(z)$ or equivalently the 
tuple $S_z$.
Indeed,
suppose 
$$S_z := \{(s_0,e_0,\beta_0),\ldots,(s_{n_1},e_{n_1},\beta_{n_1})\}$$
and
$$T_z := \{(t_0,d_0,\alpha_0),\ldots,(t_{n_1},d_{n_1},\alpha_{n_1})\}.$$
Recall that the set $S_{1z} := \{s_1,\ldots,s_{n_1}\}$
is definable uniformly in the parameter $z$.
Let $i(\gamma)$ be such that 
$\beta(\gamma) \in [s_{i(\gamma)},s'_{i(\gamma) + 1})$. 
Using that $S_{1z}$ is uniformly definable, we deduce that 
$\gamma \mapsto i(\gamma)$ is definable. 
Then by 
 Lemma \ref{T_x and S_x are intercalculable}, 
 \begin{align} 
 \rho(\gamma,h_z^{\mathrm{an}}(\beta(\gamma))) &= t_{i(\gamma) + 1} + (1/d_{i(\gamma)})(\beta(\gamma) - s_{i(\gamma)+1}) \\
 &= [(1/d_{i(\gamma)+1})(s_{i(\gamma)+1} - s_{n_x + 1}) - (1/d_{i(\gamma)+1})(\beta_{m_x} - \beta_{i(\gamma)+1})]  + (1/d_{i(\gamma)})(\beta(\gamma) - s_{i(\gamma)+1}) 
\end{align}

In the equation above, we have used the fact that 
$$\rho(h_z^{\mathrm{an}}(s_{i(\gamma)+1}),\gamma) = \rho(l_z(t_{i(\gamma)+1}),\gamma)$$
which follows from Equations (3) and (5) in the proof of Lemma \ref{T_x and S_x are intercalculable}. 
By construction, $BB_{h^{\mathrm{an}}_z,f}(z)$ depends only on 
$\gamma = h^{\mathrm{an}}(0,z)$. 
By Lemma \ref{T_x and S_x are intercalculable}, $T_z$ depends only
on $h^{\mathrm{an}}_z(0)$.

Let 
\begin{align*} 
p_X : \Sigma &\to [0,\infty] \\
          p &\mapsto [(1/d_{i(p)+1})(s_{i(p)+1} - s_{n_x + 1}) - (1/d_{i(p)+1})(\beta_{m_x} - \beta_{i(p)+1})]  + (1/d_{i(p)})(\beta(\gamma) - s_{i(p)+1})
\end{align*}

One checks using the computations above that if 
$x \in C^{\mathrm{an}}$ is a point not of type IV that retracts to $\gamma$ then 
$x \in X$  if and only if $r_X(\gamma) \leq p_X(\gamma)$. 
Given $p \in \Sigma$, let $S_p := S_z$ where 
$z \in C$ is any point that retracts to $p$. By construction, the tuple $S_p$ is well defined. 
Furthermore, $S_p$ varies definably along $\Sigma$. Hence, we have that 
$p_X$ is piecewise affine with rational slopes and parameters in $\mathrm{val}(k^*)$. 

    Note that $X$ is closed
    and the points of type I, II and III are dense in $$\{x \in C^{\mathrm{an}}| |r_\Sigma(x)| \leq p_X(h^{\mathrm{an}}(0,x))\}.$$
    Hence
     it suffices to show that the points of type I, II and III are dense in $X$. 
 If this was not true then we must have that there exists a point $q \in X$ of type IV 
 such that $q$ is isolated. This is impossible since $X$ is path connected and 
 intersects the skeleton $\Sigma$, which does not contain any points of type IV, non trivially. 
\end{proof} 
   
\section{Tameness of families of profile functions}

    Let $\phi : C' \to C$ be a finite morphism of smooth projective irreducible $k$-curves. 
Let
$h' : [0,\infty] \times C'^{\mathrm{an}} \to C'^{\mathrm{an}}$ and 
$h : [0,\infty] \times C^{\mathrm{an}} \to C^{\mathrm{an}}$ be deformation retractions 
such that the images $\Sigma'$ and 
$\Sigma$ 
are skeleta of 
$C'^{\mathrm{an}}$ and $C^{\mathrm{an}}$ respectively. 
 In \S \ref{introduction}, 
 we used the decompositions associated to 
$\Sigma'$ and $\Sigma$ to define 
a function 
$\phi_x^{\mathrm{an}} : (0,\infty] \to (0,\infty]$
for every $x \in C'(k)$. 
Theorem \ref{tameness of profile functions} implies that we can assume the functions 
$\phi_x^{\mathrm{an}}$ are determined by the retraction associated to $h'$.
In other words, we can assume that there exists some data $D_p$ for every $p \in \Sigma'$ such that 
if $x$ retracts to $p$ then $D_p$ determines $\phi_x^{\mathrm{an}}$ completely. 
    Our goal in this section is to prove Theorem \ref{tameness of profile functions in families}
    which apart from being a relative version of Theorem \ref{tameness of profile functions} 
    also implies a finiteness property for a family of morphisms $f_s : X_s \to Y_s$ of smooth projective irreducible $k$-curves 
    parametrized by a quasi-projective $k$-variety $S$.  
    More precisely, we show that there exists a family of tuples $\{(\Sigma'_s \subset X_s, \Sigma_s \subset Y_s)\}_{s \in S}$ 
    such that the family and its associated profile functions 
    are controlled by a finite simplicial complex embedded in $S^{\mathrm{an}}$. 
      Recall that we showed that profile functions can be controlled by skeleta by verifying 
    its model theoretic analogue i.e. Theorem \ref{backward branching is constant along retractions} where 
    we proved that the Backward branching index can be controlled by skeleta. 
     The primary advantage of this perspective is that it
     allows us to
      prove \ref{tameness of profile functions in families} 
    using 
     a relative version of \ref{backward branching is constant along retractions} that
    shows the required finiteness property for 
    the backward branching invariants as they vary in a family. 
    It should be noted that constructing deformation retractions for the analytification of a 
    general higher dimensional variety 
    is extremely hard and the only available tool which allows for flexibility is the
    construction of Hrushovski-Loeser. This was one of the primary motivations behind
    rephrasing the results of Temkin in \S 2. 
    \\

\noindent \emph{Notation}: \label{alternate notation}
For reasons that will be clear from the statement of 
Theorem \ref{tameness of profile functions in families}, 
we introduce the following notation. 
In place of $\phi_x^{\mathrm{an}}$ as introduced above, we write 
$\mathrm{prfl}(\phi^{\mathrm{an}},x)$.
\\

\begin{thm} \label{tameness of profile functions in families}
Let $\alpha_1 : X \to S$ and $\alpha_2 : Y \to S$ be smooth morphisms of quasi-projective $k$-varieties.  
 Let $f : X \to Y$ be a morphism such that for every $s \in S$, 
  $f_s : X_s \to Y_s$ is a finite separable morphism of smooth projective irreducible $k$-curves.
  There exists a deformation retraction 
  $g^{\mathrm{an}} : I \times S^{\mathrm{an}} \to S^{\mathrm{an}}$ whose image $\Sigma(S) \subset S^{\mathrm{an}}$ 
  is homeomorphic to a finite simplicial complex and is such that the following properties hold. 
  \begin{enumerate} 
  \item For every $s \in S(k)$, 
  there exists deformation retractions 
  $$h_s'^{\mathrm{an}} : [0,\infty] \times X_s^{\mathrm{an}} \to X_s^{\mathrm{an}}$$ and 
  $$h_s^{\mathrm{an}} : [0,\infty] \times Y_s^{\mathrm{an}} \to Y_s^{\mathrm{an}}$$
  which satisfy the assertions of Theorem \ref{tameness of profile functions}. More precisely the following hold. 
    \begin{enumerate} 
  \item The images $\Sigma_s' := {h'}_s^{\mathrm{an}}(0,X_s^{\mathrm{an}})$ and $\Sigma_s = h_s^{\mathrm{an}}(0,Y_s^{\mathrm{an}})$ are
   skeleta of the curves $X_s^{\mathrm{an}}$ and $Y_s^{\mathrm{an}}$. 
  \item We make use of the notation introduced above and Theorem \ref{tameness of profile functions}.
  The functions $\mathrm{prfl}(f_s^{\mathrm{an}},-)$ are constant along the fibres of the retraction $h'^{\mathrm{an}}_s(0,-)$.  
  Said otherwise,
  for every $x_1,x_2 \in X_s(k)$,
if  $h_s'^{\mathrm{an}}(0,x_1) = h_s'^{\mathrm{an}}(0,x_2)$ then 
$\mathrm{prfl}(f_s^{\mathrm{an}},x_1) = \mathrm{prfl}(f_s^{\mathrm{an}},x_2)$ (or equivalently
  $T^{f_s}_{x_1} = T^{f_s}_{x_2}$).
   \item There exists a function $T^{f_s} : \Sigma'_s \to \bigcup_{M \in \mathbb{N}} ((0,\infty] \times \mathbb{N})^{M}$ given by 
   $p \mapsto T^{f_s}_p$ where if
   $T^{f_s}_p  = ((t_0 = \infty,d_0,\alpha_0),\ldots,(t_{m_p},d_{m_p},\alpha_{m_p}))$ (cf. \S \ref{length of a definable path} and Theorem \ref{tameness of profile functions} (2) for notation)
  and $x$ retracts to $p$ via $h_s'^{\mathrm{an}}$ then $T^{f_s}_x = T^{f_s}_p$.
  Furthermore, the function $T^{f_s}$ is definable \footnote{By this we mean the following.
We can divide $\Sigma'_s$ into the disjoint union of line segments and points such that 
on each such subspace, the functions $p \mapsto d_i$ and 
$p \mapsto m_p$ are constant while 
$p \mapsto \alpha_i$ is piecewise affine.}
%
 \end{enumerate}  
  \item The tuple $(f_s^{\mathrm{an}} : \Sigma'_s \to \Sigma_s,T^{f_s})$ is constant along the fibres of the retraction of the deformation 
  $g^{\mathrm{an}}$ i.e. if $e$ denotes the end point of the interval $I$ and 
  if $s_1,s_2 \in S(k)$ are such that $g^{\mathrm{an}}(e,s_1) = g^{\mathrm{an}}(e,s_2)$ then 
  $\Sigma'_{s_1} = \Sigma'_{s_1}$ and $\Sigma_{s_1} = \Sigma_{s_1}$. Furthermore, with these identifications 
  $(f_{s_1}^{\mathrm{an}})_{|\Sigma'_{s_1}} = (f_{s_2}^{\mathrm{an}})_{|\Sigma'_{s_2}}$ and 
  $T^{f_{s_1}} = T^{f_{s_2}}$.  
  \end{enumerate} 
\end{thm} 
\begin{proof} 
  We use Lemma \ref{Local factorization} to deduce that there exists a finite family 
  $\{S_1,\ldots,S_m\}$ of sub-varieties of $S$ such that 
  $S = \sqcup_{1 \leq i \leq m} S_i$ and if $Y_i := Y \times_S S_i$ then
  $(\alpha_2)_{|Y_i}$ factors through a finite morphism 
  $\psi_i : Y_i \to \mathbb{P}^1 \times S_i$. 
  We can now apply 
  Theorem \ref{tameness of profile functions in families HL spaces}. 
  It follows that there exists families of deformation retractions 
  $$\{h'_s : [0,\infty] \times X_s \to \widehat{X_s}\}_{s \in S}$$ and 
  $$\{h_s : [0,\infty] \times Y_s \to \widehat{Y_s}\}_{s \in S}$$
  which are uniformly definable in the parameter $s$
   and for every $s \in S$, the images $\Upsilon_s' := h'_s(0,X_s)$ and $\Upsilon_s = h_s(0,Y_s)$ are
   skeleta of the curves $\widehat{X_s}$ and $\widehat{Y_s}$. 
   Observe from the construction in Theorem \ref{tameness of profile functions in families HL spaces} 
   that there exists a family of deformation retractions 
  $$\{h''_{is} : [0,\infty] \times \mathbb{P}^1 \times \{s\} \to \widehat{\mathbb{P}^1} \times \{s\}\}_{s \in S}$$
   which are uniformly definable in the parameter $s$
   and such that for every $s \in S$,
   $h'_s$ and $h_s$ are the unique lifts of $h''_s$. 
   As in \S \ref{homotopies of C^{an}},
   we deduce that for every $s \in S(k)$,
    there exists homotopies  
    $$h_s'^{\mathrm{an}} : [0,\infty] \times X_s^{\mathrm{an}} \to X_s^{\mathrm{an}}$$ and 
  $$h_s^{\mathrm{an}} : [0,\infty] \times Y_s^{\mathrm{an}} \to Y_s^{\mathrm{an}}$$
  whose images $\Sigma_s' := {h'}_s^{\mathrm{an}}(0,X_s^{\mathrm{an}})$ and $\Sigma_s = h_s^{\mathrm{an}}(0,Y_s^{\mathrm{an}})$ are
   skeleta of the curves $X_s^{\mathrm{an}}$ and $Y_s^{\mathrm{an}}$.
   
   Let $s \in S_i$ for some $i \in \{1,\ldots,m\}$.  
   By (1b) of Theorem \ref{tameness of profile functions in families HL spaces}, we have that 
   the functions $BB_{h'_s, \psi_{is} \circ f_s}$ and 
  $BB_{h_s, \psi_{is}}$
   are constant along the fibres of the retractions 
  $h'_s(0,-)$ and $h_s(0,-)$. 
  It follows from Lemma \ref{T_x and S_x are intercalculable} that 
  if 
  $x_1,x_2 \in X_s(k)$
 and $h_s'^{\mathrm{an}}(0,x_1) = h_s'^{\mathrm{an}}(0,x_2)$ then 
  $T^{\psi_{is} \circ f_s}_{x_1} = T^{\psi_{is} \circ f_s}_{x_2}$. Likewise, if 
$y_1,y_2 \in Y_s(k)$,
 and $h_s^{\mathrm{an}}(0,y_1) = h_s^{\mathrm{an}}(0,y_2)$ then 
  $T^{\psi_{is}}_{y_1} = T^{\psi_{is}}_{y_2}$.
  We apply the arguments from the proof of Theorem \ref{tameness of profile functions}
  to get that 
   if 
  $x_1,x_2 \in X_s(k)$
 and $h_s'^{\mathrm{an}}(0,x_1) = h_s'^{\mathrm{an}}(0,x_2)$ then 
  $T^{f_s}_{x_1} = T^{f_s}_{x_2}$.
  
     Let $g : I \times \widehat{S} \to \widehat{S}$ be the deformation retraction as provided by 
     Theorem \ref{tameness of profile functions in families HL spaces}. We have that
     the image of $g$ is a $\Gamma$-internal set.
     By \cite[Corollary 14.1.6]{HL}, 
       the deformation retraction $g$ induces a deformation retraction 
       $g^{\mathrm{an}} :  I \times {S}^{\mathrm{an}} \to {S}^{\mathrm{an}}$.
       The image $\Sigma(S)$ of $g^{\mathrm{an}}$ is homeomorphic to 
       a finite simplicial complex. 
       By (2) of Theorem \ref{tameness of profile functions in families HL spaces}, Lemma \ref{fibres lemma skeleta} and
        Lemma \ref{T_x and S_x are intercalculable} we deduce
        part (2). This concludes the proof. 
%

\end{proof}

\begin{lem} \label{Local factorization}
    Let $f : V' \to V$ be a projective morphism
    of $k$-varieties 
     such that the fibres of $f$ are pure of dimension $m$. 
    For every $v \in V(k)$, there exists a Zariski open neighbourhood $U \subset V$ of $v$ such that 
    the morphism $f : f^{-1}(U) \to U$ factors through a finite $k$-morphism 
    $p : f^{-1}(U) \to \mathbb{P}^m \times U$. 
    If $u \in U$ is such that $V'_u$ is generically reduced then the morphism 
    $p_u : V'_u \to \mathbb{P}^m \times \{u\}$ is generically étale. 
    \end{lem}      
     \begin{proof} 
     This is a relative version of \cite[Lemma 11.2.1]{HL}. 
     Let $v \in V$.
        Let 
      $n \in \mathbb{N}$ be the smallest natural number greater than or equal to $m$ such that there exists a 
      Zariski open neighbourhood $U$ of $v$ and the 
       morphism 
      $f$ factors through
       a finite morphism
      $i : V'_U \rightarrow \mathbb{P}^n \times U$.
      The fact that there exists such an $n$ is because 
      $f$ is projective. 
       Let $V'_v$ denote the fibre over $v$. If 
      $m = n$ then we have nothing to prove. Suppose $n > m$. Let 
      $(z,v) \in (\mathbb{P}^n \times V)(k)$ be a point that is not contained in $i(V'_v)$.
      Let $C := i(V') \cap (\{z\} \times V) \subset \mathbb{P}^n \times V$.
      Observe that $C$ is a closed subset of $\mathbb{P}^n \times V$ and hence $p_2(C)$ 
      is a closed subset of $V$ 
      where 
      $p_2$ is the projection $\mathbb{P}^n \times V \to V$.
       Furthermore, our choice of $z \in V$ implies that $v \notin p_2(C)$. Let
      $U$ denote the Zariski open set which is the complement of 
      $C$ in $V$.
       By construction, for every $u \in U$, $(z,u) \notin i(V'_u)$.
       Let $p : \mathbb{P}^n \smallsetminus \{z\} \to \mathbb{P}^{n-1}$
       denote the projection through the point $z$. It follows that the map 
       $p \times \mathrm{id} : (\mathbb{P}^n \smallsetminus \{z\}) \times U \to \mathbb{P}^{n-1} \times U$
       restricts to a finite morphism $i_1 : V'_U \to \mathbb{P}^{n-1} \times U$. 
       This contradicts our assumption that $n$ was minimally chosen. 
        If $u \in U$ is such that $V'_u$ is generically reduced then 
        the argument in
        \cite[\S 2.11]{dJ96} shows that $p_u : V'_u \to \mathbb{P}^m \times \{u\}$ is generically étale. 
         \end{proof}

\subsection{Theorem \ref{tameness of profile functions in families} for Hrushovski-Loeser spaces} 

     We prove a relative version of Theorem \ref{backward branching is constant along retractions}.  

\begin{prop} \label{backward branching uniformly definable} 
  Let $S$ be a quasi-projective $k$-variety and $Y := \mathbb{P}^1 \times S$. 
  Let $\alpha_1 : X \to S$ and $\alpha_2 : Y \to S$ be smooth morphisms of
   quasi-projective $k$-varieties.  
  Let $f : X_1 \to X_2$ be an $S$-morphism such that for every $s \in S$, 
  $f_s : X_{1s} \to X_{2s}$ is a finite morphism of smooth projective irreducible $k$-curves.
  Let $g : X_2 \to Y$ be an $S$-morphism such that for every $s \in S$, 
  $g_s :  X_{2s} \to Y_s$ is a finite morphism. 
 There exists a uniformly definable family 
 $\{h_s : [0,\infty] \times \mathbb{P}^1 \to \widehat{\mathbb{P}^1}\}_{s \in S}$ of deformation retractions 
 where for every $s \in S$, $h_s$ is $k(s)$-definable and the family 
satisfies the following properties. 
 \begin{enumerate}
  \item For every $s \in S$, the image of the retraction associated to $h_s$ is a $\Gamma$-internal subset of 
  $\widehat{\mathbb{P}^1} \times \{s\}$. 
  \item For every $s \in S$, the deformation retraction $h_s$ lifts uniquely to deformation retractions
  $h'_{1s} : [0,\infty] \times X_{1s} \to \widehat{X_{1s}}$ and 
  $h'_{2s} : [0,\infty] \times X_{2s} \to \widehat{X_{2s}}$
  \item For every $s \in S$, the functions $$BB_{h'_{2s},g_s} : X_{2s} \to \mathrm{Fn}([0,\infty], \mathbb{N} \times [0,\infty])$$
  and 
  $$BB_{h'_{1s},g_s \circ f_s} : X_{1s} \to \mathrm{Fn}([0,\infty], \mathbb{N} \times [0,\infty])$$
   (cf. \S \ref{backward branching index})
  are constant along the fibres of the retraction associated to
  $h'_{2s}$ and $h'_{1s}$ respectively.  
  \end{enumerate}  
\end{prop} 
\begin{proof} 
  The proof of \cite[Lemma 11.3.1]{HL} shows that there exists a constructible set $D \subset Y$ such that for 
  every $s \in S$, $D_s$ is finite, $k(s)$-definable and 
  its convex hull $\Upsilon_{D_s} \subset \widehat{\mathbb{P}^1} \times \{s\}$ 
  contains the set of forward branching points of the morphism 
  $\widehat{g_s}$ and $\widehat{g_s \circ f_s}$.  
  We fix a definable metric $m : \mathbb{P}^1 \times \mathbb{P}^1 \to \Gamma_\infty$ as in 
  \cite[\S 3.10]{HL}. 
  Let $\kappa : Y \to \Gamma_{\infty}$ be the function 
  $(y,s) \mapsto m_{D_s}(y)$ where $m_{D_s}$ is as defined 
  in \S \ref{homotopies of C}.
  
   By \cite[\S 7.4]{HL}, if $s \in S$ then any homotopy on $\widehat{\mathbb{P}^1} \times \{s\}$ that 
  respects the function $m_{D_s}$ lifts uniquely to homotopies on 
  $\widehat{X_{1s}}$ and $\widehat{X_{2s}}$.
  By \cite[Theorem 11.7.1]{HL}, there exists a uniformly definable family
   $\{h_{0s} : [0,\infty] \times \mathbb{P}^1 \to \widehat{\mathbb{P}^1}\}_{s \in S}$ 
  of deformation retractions such that for every $s \in S$, 
  $h_{0s}$ lifts uniquely to deformation retractions 
  $h'_{01s} : [0,\infty] \times X_{1s} \to \widehat{X_{1s}}$
  and 
  $h'_{02s} : [0,\infty] \times X_{2s} \to \widehat{X_{2s}}$
   whose images are $\Gamma$-internal. 
  
  Let $s \in S$. 
  In Theorem \ref{backward branching is constant along retractions}, we showed that 
  there exists a finite family of definable functions $\{\zeta_{1s} : X_{1s} \to \Gamma_\infty,\ldots,\zeta_{m_ss} : X_{1s} \to \Gamma_\infty\}$ 
  such that if $\alpha_{1s}  : X_{1s} \to [0,\infty]$ is $v+g$-continuous and $h'_{01s}[\alpha_s]$ preserves the levels of the functions
  $\zeta_{js}$ then the function $BB_{h'_{01s}[\alpha_{1s}],g_s \circ f_s}$ is constant along the retraction associated to 
  $h'_{01s}[\alpha_{1s}]$. Similarly, there exists a finite family of definable functions $\{\zeta'_{1s} : X_{2s} \to \Gamma_\infty,\ldots,\zeta'_{m'_ss} : X_{2s} \to \Gamma_\infty\}$ 
  such that if $\alpha_{2s}  : X_{2s} \to [0,\infty]$ is $v+g$-continuous and $h'_{02s}[\alpha_{2s}]$ preserves the levels of the functions
  $\zeta'_{js}$ then the function $BB_{h'_{02s}[\alpha],g_s}$ is constant along the retraction associated to 
  $h'_{02s}[\alpha_{2s}]$

   Since the families $\{h'_{01s}\}_{s \in S}$ and $\{h'_{02s}\}_{s \in S}$ are uniform,
 we can take the functions 
   $\{\zeta_{1s},\ldots,\zeta_{m_ss}\}$ and 
   $\{\zeta'_{1s},\ldots,\zeta'_{m'_ss}\}$  to be uniform in $s$. 
   The arguments in the proof 
   of Theorem \ref{radiality of definable sets} shows that for every 
   $s$, there exists definable functions 
   $\{a_{1s} : \mathbb{P}^1 \to \Gamma_\infty, \ldots, a_{n_ss} : \mathbb{P}^1 \to \Gamma_\infty\}$ and 
   $\{b_{1s} : \mathbb{P}^1 \to \Gamma_\infty, \ldots, b_{n'_ss} : \mathbb{P}^1 \to \Gamma_\infty\}$
   such that if $\alpha_s  : \mathbb{P}^1 \to [0,\infty]$ is
    $v+g$-continuous and 
   $h_{0s}[\alpha_s]$ preserves the levels of the functions
  $a_{js}$ and
  $b_{ts}$ then
  the lifts $h'_{2s} := h'_{02s}[\alpha_s \circ g_s]$
  and $h'_{1s} := h'_{01s}[\alpha_s \circ g_s \circ f_s]$
   preserve the 
  levels of the functions $\zeta_{is}$ and $\zeta'_{ts}$ and consequently $BB_{h'_{2s},g_s}$
  and $BB_{h'_{1s},g_s \circ f_s}$ 
   will be constant along the 
  fibres of the retractions
  $h'_{1s}$ and $h'_{2s}$ respectively.
   Since the given data is uniform in the parameter $s$, we 
  get that the functions $a_{1s},\ldots,a_{n_ss},b_{1s},\ldots,b_{n'_ss}$
   are uniformly definable as well
  and as a consequence, we can suppose that the family 
  $\alpha_{s}$ is uniformly definable in the parameter $s$. 
  The homotopies 
  $h_s := h_{0s}[\alpha_s]$, $h'_{1s}$ and 
  $h'_{2s}$ satisfy the assertions of the proposition. 
  \end{proof} 
  
  \begin{rem} \label{backward branching extends to the skeleton}
  \emph{Let $f : C' \to C$ be a finite morphism between smooth irreducible projective $k$-curves. 
    Let 
    $h' : [0,\infty] \times C' \to \widehat{C'}$ be a homotopy whose image $\Upsilon$
    is $\Gamma$-internal and 
    the function 
    $BB_{h',f}$ is constant along the fibres of the retraction 
    $h'(0,-)$. 
    In this case, we can define 
    $$BB_{h',f} : \Upsilon \to \mathrm{Fn}([0,\infty],\mathbb{N} \times [0,\infty])$$
     where $\mathrm{Fn}([0,\infty],\mathbb{N} \times [0,\infty])$ is the 
     set of definable functions from $[0,\infty]$ to $\mathbb{N} \times [0,\infty]$. 
   Let $p \in \Upsilon$ and $x \in C$ be a point that retracts to $p$. 
   We define $BB_{h',f}(p) := BB_{h',f}(x)$. 
   Our assumption that 
   $BB_{h',f}$ is constant along the fibres of the retraction associated 
   to $h'$ implies that the function on $\Upsilon$ is well defined.}
    \end{rem}

  \begin{thm} \label{tameness of profile functions in families HL spaces}
Let $S$ be a quasi-projective $k$-variety. 
Let $\phi_1 : X \to S$ and $\phi_2 : Y \to S$ 
be projective morphisms of quasi-projective varieties.
 Let $\phi : X \to Y$ be a morphism such that for every $s \in S$, 
  $\phi_s : X_s \to Y_s$ is a finite morphism of smooth projective irreducible $k$-curves.
  Let $\{S_1,\ldots,S_m\}$ be a family of disjoint sub-varieties of $S$ such that 
  $S = \sqcup_{1 \leq i \leq m} S_i$ and 
  for every $i$, there exists a finite morphism 
  $\psi_i : Y_i := Y \times_S S_i \to \mathbb{P}^1_{S_i}$. 
  
   We then have that there exists a deformation retraction 
  $g : I \times {S} \to \widehat{S}$ whose image $\Upsilon(S) \subset \widehat{S}$ 
  is $\Gamma$-internal and is such that the following properties hold. 
  \begin{enumerate} 
  \item There exists families of $k(s)$-definable deformation retractions 
  $$\{h'_s : [0,\infty] \times X_s \to \widehat{X_s}\}_{s \in S}$$ and 
  $$\{h_s : [0,\infty] \times Y_s \to \widehat{Y_s}\}_{s \in S}$$
  which are uniformly definable in the parameter $s$
   and
  which satisfy the following conditions.
    \begin{enumerate} 
  \item For every $s \in S$, the images $\Upsilon_s' := h'_s(0,X_s)$ and $\Upsilon_s = h_s(0,Y_s)$ are
   skeleta of the curves $\widehat{X_s}$ and $\widehat{Y_s}$. 
  \item For every $i \in \{1,\ldots,m\}$ and $s \in S_i$, the functions $BB_{h'_s, \psi_{is} \circ \phi_s}$ and 
  $BB_{h_s, \psi_{is}}$
   are constant along the fibres of the retractions 
  $h'_s(0,-)$ and $h_s(0,-)$.  
   \end{enumerate}  
  \item The tuple $(\widehat{\phi_s}_{|\Upsilon'_s} : \Upsilon'_s \to \Upsilon_s, (BB_{h'_s, \psi_{is} \circ \phi_s})_{|\Upsilon'_s}, (BB_{h_s, \psi_{is}})_{|\Upsilon_s})$ (cf. Remark \ref{backward branching extends to the skeleton})
    is constant along the fibres of the retraction of the deformation 
  $\widehat{g}$ i.e. if $e$ denotes the end point of the interval $I$ and 
  if $s_1,s_2 \in S$ are such that $g(e,s_1) = g(e,s_2)$ then 
  $\Upsilon'_{s_1} = \Upsilon'_{s_1}$ and $\Upsilon_{s_1} = \Upsilon_{s_1}$. Furthermore, with these identifications 
  \begin{align*} 
  (\widehat{\phi_{s_1}})_{|\Upsilon'_{s_1}} &= (\widehat{\phi_{s_2}})_{|\Upsilon'_{s_2}} \\
  (BB_{h'_{s_1}, \psi_{is_1} \circ \phi_{s_1}})_{|\Upsilon'_{s_1}} &= (BB_{h'_{s_2}, \psi_{is_2} \circ \phi_{s_2}})_{|\Upsilon'_{s_2}}  \\
    (BB_{h_{s_1}, \psi_{is_1}})_{|\Upsilon_{s_1}} &= (BB_{h_{s_2}, \psi_{is_2}})_{|\Upsilon_{s_2}}
    \end{align*}
  \end{enumerate} 
\end{thm} 
\begin{proof} 
 By applying Proposition \ref{backward branching uniformly definable}
 to $X \times_S S_i \to Y_i \to \mathbb{P}^1_{S_i}$ for every $i \in \{1,\ldots,m\}$, 
  we deduce that there exists uniformly definable families 
  $\{h'_s : [0,\infty] \times X_s \to \widehat{X_s}\}_{s \in S}$
and $\{h_s : [0,\infty] \times Y_s \to \widehat{Y_s}\}_{s \in S}$ 
which satisfy conditions 1(a) and 1(b) above.
In particular, we see that 
$\Upsilon' := \bigcup_s \Upsilon'_s$ and 
$\Upsilon := \bigcup_s \Upsilon_s$ are iso-definable and relatively $\Gamma$-internal subsets of 
$\widehat{X/S}$ and $\widehat{Y/S}$ respectively.
By \cite[Theorem 6.4.2]{HL}, there exists a finite pseudo-Galois covering 
$\alpha : S' \to S$ such that if $X' := X \times_S S'$ and 
$Y' := Y \times_S S'$ then there exists definable morphisms 
$\lambda' : X' \to S' \times \Gamma_\infty^N$ for some $N \in \mathbb{N}$ and
$\lambda : Y' \to S' \times \Gamma_\infty^M$ for some $M \in \mathbb{N}$ such that 
$\widehat{\lambda'} : \widehat{X'} \to \widehat{S'} \times \Gamma_\infty^N$ and 
$\widehat{\lambda} : \widehat{Y'} \to \widehat{S'} \times \Gamma_\infty^M$
are continuous morphisms that are injective when restricted to 
$\Upsilon'_1 := \Upsilon' \times_S S'$ and $\Upsilon_1 := \Upsilon \times_S S'$. 
Observe that for every $s \in S$, $\Upsilon'_s$ and $\Upsilon_s$ are 
definably compact since they are the images via continuous maps of $\widehat{X_s}$ and $\widehat{Y_s}$  
which are definably compact. 
Since $S' \to S$ is a finite morphism, we deduce that for every $s \in S'$,
$\Upsilon'_{1s}$ and $\Upsilon_{1s}$ are definably compact. 
It follows that for every $s \in S'$, 
$\widehat{\lambda'_s}$ and $\widehat{\lambda_s}$ 
restrict to homeomorphisms from $\Upsilon'_{1s}$ and 
$\Upsilon_{1s}$ onto their images in $\{s\} \times \Gamma_\infty^N$ and 
$\{s\} \times \Gamma_\infty^M$. 
Identifying $\Upsilon'_{1s}$ and $\Upsilon_{1s}$ with their images via $\lambda'_s$ and 
$\lambda_s$, we have uniformly definable families  
$\{\Upsilon'_{1s}\}_{s \in S'}$ and 
$\{\Upsilon_{1s}\}_{s \in S'}$  of subsets of $\Gamma_{\infty}^N$ and $\Gamma_{\infty}^M$
respectively.
The functions $BB_{h'_s,\psi_{is} \circ \phi_s}$
$BB_{h_s,\psi_{is}}$ lift via pull back to functions $\alpha_s$ and $\beta_s$ 
on $\Upsilon'_{1s}$ and $\Upsilon_{1s}$ (cf. Lemma \ref{backward branching extends to the skeleton}).
Furthermore, since $BB_{h'_s,\psi_{is} \circ \phi_s}$ and 
$BB_{h_s,\psi_{is}}$ are uniformly definable, we see that 
$\alpha_s$ and $\beta_s$ are uniformly definable for $s \in S'$. 
Recall that for every $s \in S$, 
$BB_{h'_s,\psi_{is} \circ \phi_s}$ 
 is completely determined by a finite collection 
of $\Gamma_\infty$-valued definable functions which can be assumed to be uniform in the parameter $s$. 
The pull back of these functions to $\Upsilon'_{1s}$ must then completely 
determine $\alpha_s$. 
The same holds for the collection $\beta_s$. 

We use the arguments in 
\cite[Theorem 6.4.4]{HL}, to get a finite family 
of $\Gamma_\infty$-valued definable functions 
$\xi_j$ on $S'$ such that for every $s \in S'$, 
$\{\xi_j(s)\}_j$ completely determine 
$\Upsilon'_{1s}$, $\Upsilon_{1s}$, the functions $\alpha_s$ and 
$\beta_s$ and the morphism $\Upsilon'_{1s} \to \Upsilon_{1s}$.  
More precisely, if for $s,s' \in S'$, 
$\xi_j(s) = \xi_j(s')$ for every $j$ then 
$\Upsilon'_{1s} = \Upsilon'_{1s'}$, 
$\Upsilon_{1s} = \Upsilon_{1s'}$ and with these identifications 
$\Upsilon'_{1s'} \to \Upsilon_{1s'}$ coincides with 
$\Upsilon'_{1s} \to \Upsilon_{1s}$, 
$\alpha_{s'} = \alpha_s$ and 
$\beta_{s'} = \beta_s$. 
    By \cite[Theorem 11.1.1]{HL}, there exists a deformation 
    retraction 
   $g' : I \times S' \to \widehat{S'}$ whose image is $\Gamma$-internal, 
   the functions $\xi_j$ are constant along the fibres of the retraction
   and $g'$ is $\mathrm{Aut}(S'/S)$-invariant. 
   It follows that $g'$ descends to a deformation retraction 
   $g : I \times S \to \widehat{S}$ whose image is $\Gamma$-internal. 
   One checks that condition (2) of the statement is satisfied for the
   retraction associated to $g$.

\end{proof}

           \bibliographystyle{plain}
\bibliography{library}

 \noindent  John Welliaveetil \\
 Department of Mathematics,
 Imperial College London, \\
 180 Queens Gate, \\
 SW7 2AZ, \\
 London, UK. \\
 \textit{email : welliaveetil@gmail.com}
   \end{document}